\documentclass[11pt]{article}

\usepackage[margin=1in]{geometry}
\usepackage{amsmath,amssymb,amsthm,amsfonts}
\usepackage{graphicx}
\usepackage{epstopdf}
\usepackage{xcolor}
\usepackage{algorithm}
\usepackage{algorithmic}
\usepackage{booktabs}
\usepackage{microtype}
\usepackage[numbers,sort&compress]{natbib}
\usepackage[hidelinks]{hyperref}

\DeclareGraphicsExtensions{.pdf,.eps,.png,.jpg}

\theoremstyle{plain}
\newtheorem{theorem}{Theorem}[section]
\newtheorem{lemma}[theorem]{Lemma}

\newtheorem{definition}[theorem]{Definition}
\newtheorem{assumption}[theorem]{Assumption}
\newtheorem{example}[theorem]{Example}

\theoremstyle{remark}
\newtheorem{remark}[theorem]{Remark}

\usepackage[capitalize,nameinlink]{cleveref}

\crefname{assumption}{Assumption}{Assumptions}
\crefname{example}{Example}{Examples}
\crefname{property}{Property}{Properties}
\crefname{remark}{Remark}{Remarks}

\newcommand{\be}{\begin{equation}}
\newcommand{\ee}{\end{equation}}
\newcommand{\tr}{\operatorname{tr}}
\newcommand{\bm}{\boldsymbol}

\DeclareMathOperator*{\argmin}{\arg\min}

\def\calA{\mathcal{A}}
\def\calC{\mathcal{C}}
\def\calH{\mathcal{H}}

\def\calO{\mathcal{O}}

\def\calS{\mathcal{S}}

\def\spec{S}
\def\simp{\Delta}

\def\Conv{\mbox{Conv}}
\def\dist{{\rm dist}}
\def\diag{{\rm diag}}

\def\sgn{{\rm sgn}}
\def\SO{{\rm SO}}

\def\Xstar{\mathcal{X}^*}

\def\bfone{{\bm 1}}
\def\bfa{{\bm a}}
\def\bfb{{\bm b}}

\def\bfe{{\bm e}}

\def\bfv{{\bm v}}
\def\bfw{{\bm w}}
\def\bfx{{\bm x}}
\def\bfy{{\bm y}}

\def\bflambda{\bm{\lambda}}

\def\tileta{{\tilde{\eta}}}
\def\tilS{{\tilde{S}}}

\title{Rank-Adaptive and Linearly Convergent Frank--Wolfe Method over Spectrahedron via Nonconvex Oracle}

\author{%
Houduo Qi$^{\dagger}$ \qquad
Haoning Wang$^{\ddagger}$ \qquad
Liping Zhang$^{\ddagger}$\\[0.85em]
\normalsize
$^{\dagger}$Department of Data Science and Artificial Intelligence, and Department of Applied Mathematics,\\
The Hong Kong Polytechnic University, Hong Kong\\
\texttt{houduo.qi@polyu.edu.hk}\\[0.45em]
$^{\ddagger}$Department of Mathematical Sciences, Tsinghua University, Beijing 100084, China\\
\texttt{whn22@mails.tsinghua.edu.cn},\quad
\texttt{lipingzhang@tsinghua.edu.cn}%
}

\date{}

\hypersetup{
  pdftitle={Rank-Adaptive and Linearly Convergent Frank--Wolfe Method over Spectrahedron via Nonconvex Oracle},
  pdfauthor={H. Qi, H. Wang, and L. Zhang}
}

\begin{document}

\maketitle

\begin{abstract}
For Frank--Wolfe (FW) methods for convex optimization over the spectrahedron, it remains open whether a block-update variant can be linearly convergent when the update rank never exceeds the (unknown) optimal rank~$r^*$ at each iteration.
Existing block and spectral FW methods require an update rank at least~$r^*$ (typically prior knowledge of~$r^*$) to obtain a linear rate.
This paper develops a rank-adaptive FW method whose update rank never exceeds $r^*$
 %satisfy $\widehat{r}_t\le r_t\le r^*$ 
at every iteration and which converges linearly after a finite burn-in under quadratic growth and strict complementarity, the two conditions commonly used in spectral FW analyses.
The method is built on two designs.
The first is a nonconvex spectral oracle, motivated by the geometric connection between the simplex and the spectrahedron; it yields a thresholding rank~$r_t$ of the current iterate and a closed-form low-rank solution.
Computing~$r_t$ exactly, however, requires a full eigendecomposition.
The second introduces the efficient rank~$\widehat{r}_t$ of the current iterate, a cheap surrogate that inherits the optimality properties of the spectral oracle.
The algorithm switches between the thresholding rank and the efficient rank so that the actual FW update uses~$\widehat{r}_t$, keeps the per-iteration cost comparable to standard FW, and eventually identifies~$r^*$.
These results close the gap between low-rank efficiency and fast convergence for Frank--Wolfe methods over the spectrahedron.
Numerical experiments demonstrate the advantage of the proposed method.
\end{abstract}


\vspace{0.5em}

\section{Introduction}

We study efficient first-order methods for solving the following convex problem:
\begin{equation}\label{Eq: Problem}
    \min_{X \in \spec_n} f(X) \quad \mbox{s.t.} \ \  
    X \in \spec_n := \{ X \in \mathbb{S}^n \mid \tr(X) = 1,\, X \succeq 0 \},
\end{equation}
where $\spec_n$ denotes the \emph{spectrahedron} defined over the space $\mathbb{S}^n$ of $n \times n$ real symmetric matrices. 
Here, $X \succeq 0$ indicates that $X$ is positive semi-definite. 
The objective function $f : \mathbb{S}^n \to \mathbb{R}$ is assumed to be convex and $L$–smooth.
% 介绍这类问题的重要性
Problem~\eqref{Eq: Problem} serves as a fundamental formulation underlying a broad class of low-rank matrix recovery tasks.
%\footnote{Minimizing a convex function over a nuclear-norm ball can be efficiently reduced to the same problem over the spectrahedron; see \cite{jaggi2010simple} for details.}  
Notable examples include matrix completion~\cite{jaggi2010simple, candes2012exact}, distance metric learning~\cite{xing2002distance, ying2012distance}, matrix sensing~\cite{yurtsever2017sketchy,recht2010guaranteed}, and blind deconvolution~\cite{ahmed2013blind}, among many others~\cite{dudik2012lifted, zhang2012accelerated, braun2025conditional}.
This paper studies Frank--Wolfe (FW) methods for~\eqref{Eq: Problem} and develops a rank-adaptive variant that never exceeds the optimal rank~$r^*$, needs no a priori estimate of it, and is linearly convergent under quadratic growth and strict complementarity.

% 介绍为什么研究FW类算法
\subsection{Literature review}

We focus on methods built upon the classical FW framework~\cite{frank1956algorithm, levitin1966constrained}.
FW avoids the $O(n^3)$ projection onto the spectrahedron: its linear minimization oracle (LMO) needs only a leading eigenvector of the gradient, which iterative eigensolvers typically obtain at a cost that scales almost linearly with the number of nonzeros.
The updates remain low-rank, so quantities such as $f(X)$ or $\nabla f(X)$ can be maintained incrementally---a substantial saving when the objective involves $g(\calA(X))$, $g(\calA(X^{-1}))$, or a $\log\det$ term~\cite{yurtsever2017sketchy,danon2022frank}.
Finally, FW needs little parameter tuning and naturally admits line search or backtracking~\cite{pedregosa2020linearly}.

A well-known limitation of the classical FW method is its $O(1/\epsilon)$ worst-case convergence rate.  
This rate does not improve even under standard curvature assumptions such as strong convexity or quadratic growth~\cite{jaggi2013revisiting, lan2013complexity}, whereas projection-based first-order methods achieve a linear convergence rate of order $O(\log(1/\epsilon))$. 
To close this gap, a growing line of recent work has explored how to accelerate FW-type algorithms on the spectrahedron under additional structural assumptions, including quadratic growth and 
strict complementarity.
In particular, the rank of the block-update appears to play a central role in
achieving the acceleration. We briefly comment on two groups of research
along this line. 

{\bf (A) Rank-1 update.}
One notable development is the Regularized FW method of~\cite{garber2016faster}, which relies solely on leading-eigenvector oracles and, under strong convexity, attains an $O(1/\sqrt{\epsilon})$ rate in expectation---the first acceleration of FW over the spectrahedron, though short of linear convergence.
When the optimum is rank-one and satisfies strict complementarity,~\cite{garber2023linear} further shows that standard FW with line search is linearly convergent after a finite burn-in; the argument does not extend to a higher-rank optimum.
A subsequent line~\cite{garber2025linearly,garber2026linearconvergencefrankwolfetypemethod} combines Frank--Wolfe steps with specialized away and pairwise updates, still using only extreme-eigenvector oracles.
The randomized SAP-FW method of~\cite{garber2025linearly} attains linear convergence in expectation after a burn-in under quadratic growth and strict complementarity.
It requires the smoothness constant, performs three extreme-eigenvector computations per iteration, and maintains a thin factorization (or the pseudoinverse) of $X_t$ to implement the away and pairwise steps; a stable implementation is nontrivial.
The follow-up~\cite{garber2026linearconvergencefrankwolfetypemethod} relaxes the strict complementarity requirement of SAP-FW and, within the same away--pairwise framework, proposes a deterministic linearly convergent variant.
A limitation of the resulting guarantee is that it can become arbitrarily slow when the linear term of the objective grows.
On the quadratic $f(X)=\frac12\lVert X-t\,\mathrm{diag}\{1,-1\}\rVert_F^2$ as $t\to\infty$, the complexity guarantees of classical FW and of the linearly convergent variants reviewed above do not deteriorate with~$t$, whereas the linear rate proved in~\cite{garber2026linearconvergencefrankwolfetypemethod} tends to zero.

{\bf (B) Rank-r update}.
Another line of work achieves linear convergence by computing the top-$r$ eigen-components at each iteration, rather than relying solely on rank-one updates. 
This approach can be viewed as a middle ground between projection-based methods and classical FW algorithms.  
The Block-FW method introduced in~\cite{allen2017linear} attains linear convergence under quadratic growth and strict complementarity assumptions, provided that the block size satisfies $r \geq r^*$, where $r^*$ denotes the rank of an optimal solution. 
The numerical benefits by the block-updates with $r>1$ are also well explained and justified in \cite{ding2020k}. 
However, when $r < r^*$, the method lacks convergence guarantees, which limits its robustness in practice.  
To overcome this limitation, \cite{ding2020spectral} proposed the Spectral Frank-Wolfe method, which matches the sublinear convergence rate of standard FW when $r < r^*$ and achieves linear convergence once $r$ exceeds the optimal rank, under the same assumptions.  
A common drawback of both approaches is their reliance on a tight estimate of $r^*$, which is typically unavailable in applications; choosing $r$ larger than necessary leads to unnecessarily expensive partial SVD computations.
Consequently, it remains open whether a block-update FW method can attain a linear rate when the update rank never exceeds the unknown~$r^*$.
The analyses in~\cite{allen2017linear, ding2020spectral} appear to leave little room for improvement within their current update rules when $r<r^*$; new oracles are needed to tackle this case.
Table~\ref{table: alg-compare} summarizes the above FW-type methods in terms of the eigen-decomposition rank, required parameters, burn-in requirements, and convergence rates.

\begin{table}[htbp]
\centering
\footnotesize
\caption{Comparison of FW-type methods for the spectrahedron.
``EV rank'' is the number of eigen-components computed per iteration;
``Rate'' is the convergence rate after any burn-in phase.
Entries marked ``(E)'' hold in expectation.}
\label{table: alg-compare}
\begin{tabular}{@{}lcccc@{}}
\toprule
Algorithm & EV rank & Parameters & Burn-in & Rate \\
\midrule
Frank--Wolfe~\cite{frank1956algorithm}
  & $1$ & --- & No & $1/t$ \\
Regularized-FW~\cite{garber2016faster}
  & $1$ & $L$ & No & $1/t^2$ (E) \\
SAP-FW~\cite{garber2025linearly}
  & $1$ & $L$ & Yes & linear (E) \\
Block-FW~\cite{allen2017linear}
  & $r\,(\ge r^*)$ & $r\ge r^*,\,\mu,\,L$ & No & linear \\
Spectral-FW~\cite{ding2020spectral}
  & $r\,(\ge r^*)$ & $r\ge r^*$ & Yes & linear \\
This paper
  & $\widehat{r}_t\,(\le r^*)$ & $\mu$ & Yes & linear \\
\bottomrule
\end{tabular}
\end{table}

%%%%
\subsection{New proposal and main contributions}

As already indicated in Table~\ref{table: alg-compare}, our proposal addresses the bottleneck of achieving a linear rate with an update rank not exceeding~$r^*$.
The method is built on two designs.
The first is a nonconvex \emph{Spectral Oracle} based on the spectral simplex ball; it produces a \emph{thresholding rank}~$r_t$ of the current iterate and a closed-form low-rank solution.
Computing~$r_t$ exactly, however, requires a full eigenvalue decomposition.
The second design introduces an \emph{efficient rank}~$\widehat{r}_t$, a cheap surrogate that inherits the optimality properties of the Spectral Oracle.
The algorithm switches between the two ranks so that the actual FW update uses~$\widehat{r}_t$ and satisfies $\widehat{r}_t\le r_t\le r^*$ at every iteration.
This leads to a \emph{linearly convergent} and \emph{rank-adaptive} Frank--Wolfe method for~\eqref{Eq: Problem} that requires no prior knowledge of~$r^*$ under quadratic growth and strict complementarity.
We postpone all the technical developments to the remaining sections and
summarize the main contributions below.

\begin{enumerate}
    \item \textbf{Spectral ball and nonconvex Spectral Oracle.}
    Motivated by the simplex ball framework of~\cite{wang2025simplex}, which yields linear convergence over polytopes via a novel oracle, we lift the underlying geometry to the spectrahedron.
    The key observation is that the eigenvalue vector of any $X\in\spec_n$ lies in the unit simplex $\simp_n=\{x\in\mathbb{R}^n:x_i\ge 0,\,\sum_{i=1}^n x_i=1\}$.
    We therefore define a \emph{spectral ball} as the set of matrices whose eigenvalues belong to a prescribed simplex ball, and introduce a \emph{Spectral Oracle} by restricting the Frank--Wolfe linear minimization step to the intersection of this ball with spectrahedron.
    Despite the nonconvexity of this constraint set, the oracle admits a closed-form low-rank solution whose rank is the \emph{thresholding rank} of the current iterate and whose range is a principal eigenspace of the gradient.

    \item \textbf{Rank adaptivity without knowing $r^*$.}
    Embedding the Spectral Oracle into the Frank--Wolfe framework yields a rank-adaptive scheme (ERSO-FW) in which the number of eigen-components computed at iteration~$t$ is the \emph{efficient rank}~$\widehat{r}_t$, a cheap surrogate of the thresholding rank~$r_t$.
    We prove that $\widehat{r}_t\le r_t\le r^*$ for all~$t$ and that, after finitely many iterations, both ranks stabilize exactly at the optimal rank~$r^*$.
    Consequently, the method needs no a priori estimate of $r^*$ and can recover it by monitoring the sequence $\{\widehat{r}_t\}$.

    \item \textbf{Linear convergence with low per-iteration cost.}
    Building on this rank-adaptive oracle, we develop a fully corrective variant that attains linear convergence after a finite burn-in phase under quadratic growth and strict complementarity.
    Both the burn-in length and the linear rate depend only on the smoothness constant~$L$, the quadratic-growth constant~$\mu$, and the strict-complementarity quantities~$\delta$ and~$\lambda_{r^*}$ (see Section~\ref{Subsection-Notation-Assumpt}), and are independent of the ambient dimension~$n$; the linear rate is moreover independent of the optimal rank~$r^*$, whereas the burn-in length involves an additional additive~$r^*$ term caused by the use of the efficient rank.
    Aside from computing the top-$\widehat{r}_t$ eigen-components, implementing the Spectral Oracle requires only $O(r^*n^2)$ additional work per iteration, so the cost remains comparable to standard Frank--Wolfe methods when $r^*\ll n$.
    Numerical experiments on quadratic sensing, matrix completion, and polynomial neural network training support the theory and demonstrate competitive, robust performance.
\end{enumerate}

\subsection{Organization}

The remainder of this paper is organized as follows. 
Section~\ref{section-preliminaries} introduces the notation and assumptions used throughout the paper, and reviews the simplex ball and its key properties. 
Section~\ref{section-spectral-ball-oracle} presents the spectral ball, studies its structural properties, and introduces the Spectral Oracle along with an efficient implementation. 
Section~\ref{section-main-results} develops our main algorithms, including a natural Frank--Wolfe variant based on the Efficient-Rank Spectral Oracle and a fully corrective variant that achieves linear convergence. 
Section~\ref{section-numerical-tests} reports numerical experiments to validate the effectiveness of the proposed methods.
We conclude in Section~\ref{sec:conclusions}.

\section{Preliminaries}\label{section-preliminaries}

\subsection{Notation and assumptions}\label{Subsection-Notation-Assumpt}
We use lower-case letters, bold lower-case letters, and capital letters to denote scalars, vectors, and matrices, respectively (e.g., $x$, $\bfx$, and $X$). 
The Euclidean norm for vectors is denoted by $\|\cdot\|$, and the standard inner product by $\langle \cdot, \cdot \rangle$. 
For a vector $\bfx \in \mathbb{R}^n$, its $i$-th component is denoted by $x_i$ or $x(i)$. 
The vector $\bfone_n$ is the all-ones vector, and $\bfe_i$ is the $i$-th standard basis vector in $\mathbb{R}^n$, with a $1$ in the $i$-th position and $0$ elsewhere. 
For vectors $\bfx, \bfy \in \mathbb{R}^n$, $\max\{\bfx, \bfy\}$ denotes the component-wise maximum, forming a new vector with entries $\max\{x_i, y_i\}$.
When $\bfy=0$, we denote $\max\{ \bfx, 0\} = \bfx_+$. 
For $\bfx \in \mathbb{R}$, $\sgn(\bfx)$ is the sign vector of $\bfx$
and $\| \bfx\|_0$ is the zero norm that counts the nonzero elements in $\bfx$.
The Hadamard product $\bfx \circ \bfy$ is the componentwise product vector
of $\bfx$ and $\bfy$.
The vector $\min\{\bfx, \bfy\}$ is defined similarly.
For a vector $\bfx \in \mathbb{R}^n$, a subset $C \subseteq \mathbb{R}^n$, and $\tau > 0$, we define the distance $\dist(\bfx, C) := \min_{\bfy \in C} \lVert \bfx - \bfy \rVert$, the Minkowski sum $\bfx + C := \{ \bfx + \bfy \mid \bfy \in C \}$, and the scaled set $\tau C := \{ \tau \bfy \mid \bfy \in C \}$. 
The convex hull of a set $\mathcal{V}$ is denoted by $\Conv\{\mathcal{V}\}$. 
For any positive integer $n$, the set $[n] := \{1, \dots, n\}$.

For matrices, the spectral norm and Frobenius norm are denoted by $\|\cdot\|$ and $\|\cdot\|_F$, respectively. 
The inner product $\langle \cdot, \cdot \rangle$ on symmetric matrices is the trace inner product. 
For a matrix $A \in \mathbb{S}^n$, its eigenvalues in non-ascending order are denoted by $\lambda_1(A) \ge \dots \ge \lambda_n(A)$, and the eigenvalue vector by $\bflambda(A) = (\lambda_1(A), \dots, \lambda_n(A))^T$.
For convenience, we let $\lambda_{\max}(A) = \lambda_1(A)$ and
$\lambda_{\min}(A) = \lambda_n(A)$.
We further denote by $\bflambda_{\uparrow}(A) := (\lambda_n(A), \dots, \lambda_1(A))^T$ the same vector arranged in non-decreasing order.
We denote by $\mbox{EV}_{k}(A)$ eigenvectors of A that corresponds to the $k$-smallest (signed) eigenvalues of $A$.
For a vector $\bflambda \in \mathbb{R}^n$, $\diag(\bflambda)$ denotes the diagonal matrix with $\bflambda$ on its diagonal. 
The $n$-dimensional identity matrix is denoted by $I_n$, and the orthogonal group by $\mathcal{O}_n := \{ P \in \mathbb{R}^{n \times n} : P P^T = I_n \}$. 
For any integer $r>0$, let $S_r := \{ S \in \calS^r \ | \ \tr(S) = 1, \ S \succeq 0\}$.

\subsection{Basic assumptions}

We let $\Xstar \subseteq S_n$ denote the optimal solution set of
% $\min_{X \in S_n} f(X)$, 
\eqref{Eq: Problem}
with $f^*$ denoting the optimal value. 
We assume that $f$ is convex and $L$-smooth over $S_n$, i.e., for all $X, Y \in S_n$,
\[
\lVert \nabla f(X) - \nabla f(Y) \rVert_F \leq L \lVert X - Y \rVert_F.
\]
This implies the standard inequality (see, e.g., \cite[Thm 2.1.5]{nesterov2013introductory}):
\[
f(Y) \leq f(X) + \langle \nabla f(X), Y - X \rangle + \frac{L}{2} \lVert X - Y \rVert_F^2, \quad \forall X, Y \in S_n.
\]

We now introduce our two main assumptions, quadratic growth and strict complementarity, which hold throughout this work. Quadratic growth is a standard condition in the literature on linear convergence rates for first-order methods \cite{necoara2019linear,ding2020spectral}.

\begin{assumption}[Quadratic Growth (QG)]
    There exists $\mu > 0$ such that for all $X \in S_n$,
    \begin{equation}
    \dist^2(X, \Xstar) \leq \frac{2}{\mu} (f(X) - f^*).
    \end{equation}
\end{assumption}

Strict complementarity is key to linear rates for Frank--Wolfe methods; on the spectrahedron it amounts to a positive eigengap of $\nabla f$ at optima~\cite{ding2020spectral,garber2023linear,garber2025linearly}.

\begin{assumption}[Strict Complementarity]\label{Assum-strict-comp}
    There exists an integer $r^* \in [n]$ such that for all $X^* \in \Xstar$, $\mbox{rank}(X^*) = r^*$.
    Let $\lambda_{r^*}:=\min_{X\in \Xstar} \lambda_{r^*}(X^*)>0$ denote the smallest nonzero eigenvalue among all optimal solutions.
    If $r^* < n$, there exists $\delta > 0$ such that
    \begin{equation}
    \lambda_{n-r^*}(\nabla f(X^*)) - \lambda_{n-r^*+1}(\nabla f(X^*)) = \delta, \quad \forall X^* \in \Xstar.
    \end{equation}
\end{assumption}
Note that Assumption~\ref{Assum-strict-comp} implies that $\Xstar$ lies on an $r^*$-dimensional face of $S_n$.
For any $A \in \calS^n$ and any integer $r \in [n]$, we call
\[
 \calC_r(A) := \left\{  
  V_r S V_r^\top \ | \ S \in S_r
 \right\}
\]
the spectral $r$th set of $A$, where $V_r \in \mathbb{R}^{n \times r}$ consists of the orthonormal
eigenvectors corresponding to the $r$ smallest eigenvalues of $A$.
The following result plays an important role in our analysis.

\begin{lemma}\cite[Lemma~5]{ding2020spectral}
\label{Lemma-Ding}
Given $Y \in \calS^n$ which satisfies $\lambda_{n-r}(Y) - \lambda_{n-r+1}(Y) \ge \delta$ for some $\delta>0$,
then for any $X \in S_n$, there is some $W \in \calC_r(Y)$ such that
\[
 \langle X-W, \; Y\rangle \ge \frac{\delta}2 \| X - W\|^2_F .
\]
\end{lemma}

\subsection{Simplex ball} \label{subsection-simplex-ball}

As discussed earlier, our approach is inspired by the Simplex FW method \cite{wang2025simplex}, 
which achieves linear convergence over the simplex via a novel oracle design. 
For completeness, we briefly review the definition of the simplex ball and its key properties, 
which will serve as the foundation for our extension to the spectrahedral setting.

\begin{definition}[Simplex ball]\label{Def: simplex_ball}
	Let
	$
	  \Delta_0 := \Delta_n - \frac 1n \bfone_n .
	$    
	For any $\bfx\in\mathbb{R}^n$ and $d>0$, we define $\Delta(\bfx, d)$ as the simplex ball of radius $d$ centered at $\bfx$ by
    \begin{equation}\label{SimplexBall-New}
    	\Delta(\bfx, d) := \bfx + (nd) \Delta_0 = \Big\{ (\bfx - d\bfone_n) + nd\bflambda \mid \bflambda \in \Delta_n  \Big\}.
       % \Delta(\bfx,d):=\{\bfx+d\bfr\vert \bfr\in\text{conv}\{n\bfe_i-\bfone_n:i\in [n] \} \}.
    \end{equation}
\end{definition}
The following properties of the simplex ball are crucial to our development.

\begin{lemma}\cite[Lemma~1]{wang2025simplex}
\label{Lemma: Simplex ball}
    Given $\bfx\in \Delta_n$ and $d>0$, we have the following statements.
    \begin{enumerate}
        \item[{\rm (1)}] The unit simplex is a simplex ball, i.e., $\Delta_n = \Delta(\frac{1}{n}\bfone_n, \frac{1}{n})$.
%        Moreover, we have
%        \begin{equation}\label{SimplexBall}
%        	 \Delta(\bfx,d) =\Big\{\bfx+d\bfr\vert \bfr\in \rm{Conv}\{n\bfe_i-\bfone_n:i\in [n] \} \Big\}.
%        \end{equation}

        \item[{\rm (2)}]
        The intersection of two simplex balls, if nonempty, is again a simplex ball. In particular,
        \be \label{Eq: definition_d}
         \Delta_n\cap \Delta(\bfx,d)=\Delta(\bar{\bfx},\bar{d}) \ \
         \mbox{where} \ \
         \left\{
         \begin{array}{l}
         	\bar{d} = \frac{1}{n}\sum_{i=1}^n\min\{d, x_i\} \\ [0.2ex]
         	\bar{\bfx} = \max\{\bfx, d\bfone_n\}+(\bar{d}-d)\bfone_n.
         \end{array}
         \right .
        \ee

%        Moreover, for $\bfx_1,\bfx_2\in \Delta_n$ and radius $d_1,d_2>0$ such that
%         $\Delta(\bfx_1,d_1)\cap \Delta(\bfx_2,d_2)\neq\emptyset$, it holds
%         \[
%         \Delta(\bfx_1, d_1)\cap \Delta(\bfx_2,d_2)=\Delta(\bfx_3, d_3),
%         \]
%         where
%         \begin{equation}\label{Eq: intsect_simplex_balls}
%         \left\{
%         	\begin{aligned}
%         		&d_3   = \frac{1}{n}(1+\sum_{i=1}^n\min\{d_1-x_1(i),d_2-x_2(i) \}),\\
%         		&x_3(i)  = \max\{x_1(i)-d_1,x_2(i)-d_2 \}+ d_3, \ \ i\in[n].
%         	\end{aligned}
%         \right.
%         \end{equation}
%         Consequently, we have $d_3 \le \min\{d_1, d_2\}$.

%        \item[{\rm (3)}]  The linear optimization over a simplex ball has the following closed-form solution:
%        \begin{align*}
%          & \bfy^* := \bfx+ (nd) ( \bfe_{i^*}- \bfone_n/n ) \in  \argmin_{\bfy\in \Delta(\bfx,d)}\ \langle \bfc, \bfy\rangle \ \ \mbox{with} \ \ i^*=\argmin_{i\in [n]} c_i.
%        \end{align*}

        \item[{\rm (3)}] For any point $\bfy\in \Delta_n$, if $\lVert \bfx-\bfy\rVert\leq d$, then $\bfy\in \Delta(\bfx,d)$. 
    \end{enumerate}
\end{lemma}

\section{Spectral Simplex ball and spectral oracle}
\label{section-spectral-ball-oracle} 
In this section, we formally introduce the spectral ball and explore its key properties. Subsequently, we present the Spectral Oracle (SO) based on spectral ball and propose an efficient algorithm for its computation. 

\subsection{Spectral Simplex ball and its properties}
\label{subsection-spectral-ball}
Motivated by the relationship between the simplex and the spectrahedron, we define the spectral ball by extending the concept of the simplex ball. The formal definition is as follows. 
\begin{definition}[Spectral Simplex ball]
For any $X\in S_n$ and $d> 0$, we define $S(X,d)$ as the spectral ball of radius $d$ centered at $X$ by
\begin{equation}
    S(X,d):=\{Q\diag(\bflambda)Q^T: \ Q\in \calO_n, \ \bflambda\in \simp(\bflambda(X),d) \}.
\end{equation}
\end{definition}

Although the spectral ball is a natural extension of the simplex ball, its geometric structure differs markedly from both the simplex ball and the spectrahedron. 
Unlike the simplex ball and spectrahedron, which are convex sets, the spectral ball is generally {non-convex}. 
Furthermore, the set-valued mapping $S: S_n \times \mathbb{R}_+ \rightrightarrows S_n$ induced by the spectral ball is not injective, in contrast to the simplex ball.
A particularly interesting feature is that even as the radius $d \to 0$, the diameter of $S(X, d)$ does not necessarily vanish, indicating that the spectral ball may retain a nontrivial extent in the limit.
The following example illustrates these properties.

\begin{example} \label{Example-Nonconvexity}
Let
\[
    X_1=\left(\begin{array}{cc}
        1 & 0 \\
        0 & 0
    \end{array} \right),\qquad 
    X_2=\left(\begin{array}{cc}
        0 & 0 \\
        0 & 1
    \end{array} \right).
\]
We can verify that
\begin{align*}
 S(X_1, d) = S(X_2, d) &=
 \left\{
  Q \begin{pmatrix}
  	1- \beta & \\
  	 & \beta 
  \end{pmatrix} Q \ \Big|\ \ Q \in \calO_2, \ -d \le \beta \le d  
 \right\} \\
 &= \left\{
  X \in \calS^2 \ \Big| \ \tr(X)=1, \ \lambda_{\max}(X) \le 1+d , \
                                      \lambda_{\min}(X) \le d
 \right\}.
\end{align*}
It is easy to see that $X_1, X_2 \in S(X_1, d) = S(X_2, d)$. However, $X_3 := \frac{1}{2}(X_1 + X_2) = \frac{I_2}{2}$ does not belong to $S(X_1, d)$
when $d < 1/2$. 
This example demonstrates three key facts.  
\begin{itemize}
    \item \textbf{Non-convexity}: the convex combination $X_3$ lies outside $S(X_1, d)$.
 The nonconvexity is also due to the constraint $\lambda_{\min}(X) \le d$ and the smallest eigenvalue function is nonconvex.
   
    \item \textbf{Non-injectivity}: distinct centers $X_1$ and $X_2$ yield the same spectral ball, i.e., $S(X_1, d) = S(X_2, d)$.
    
    \item \textbf{Non-vanishing diameter}: even when $d$ is arbitrarily small, the distance $\|X_1 - X_2\|_F = \sqrt{2}$ shows that the spectral ball retains a positive diameter as $d \to 0$.
\end{itemize}
\end{example}

Despite these differences, the spectral ball enjoys some very nice
properties.
%shares several key properties with the simplex ball, which are essential for our analysis.

\begin{lemma}\label{Lemma: Spectral-ball}
Given $X\in S_n$ and $d>0$, we have
\begin{enumerate}
    \item[(1)] The spectrahedron is a spectral ball, i.e., $S_n=S(I_n/n,1/n)$.
    \item[(2)] 
    Suppose $d < \lambda_1(X)$.
    The intersection of the spectrahedron and a spectral ball $S_n\cap S(X,d)$, is again a spectral ball $S(\overline{X},\overline{d})$, where
    \begin{equation}\label{Eq-spectral-d}
    \left\{
     \begin{array}{l}
         \overline{d}=\frac{1}{n}(1+rd-\sum_{i=1}^r\lambda_i(X)) \\ [0.6ex]
         \overline{X}=\overline{d}I_n+\diag\{\lambda_1(X)-d,\dots,\lambda_r(X)-d,0,\dots,0\},\\ 
     \end{array} 
     \right.
    \end{equation}
    with $r := \max \{ i \in [n] \mid \lambda_i(X) > d \}$.
    \item[(3)] Given a matrix $G \in \calS^n$,  
    the linear optimization 
    \[
       \min_{Y\in S(X,d)}\langle G,Y\rangle
    \]
    has a closed-form solution:
    \[
        Y^*=P \diag(\bflambda(X) - d \mathbf{1}_n + n d \mathbf{e}_1) P^T , 
    \]
    where $P^TGP = \diag(\bflambda_{\uparrow}(G))$.
     
    \item[(4)] For every $Y\in S_n$, if $\|X-Y \|_F\leq d$, then $Y\in S(X,d)$. 
    
\end{enumerate}
\end{lemma}

\begin{proof}
\begin{enumerate}
    \item[(1)] By the definition of the spectral ball and Lemma~\ref{Lemma: Simplex ball}(1),
    \[
        \begin{aligned}
        S(I_n/n, 1/n) &= \{ Q \diag(\lambda) Q^T : Q \in \mathcal{O}_n, \lambda \in \simp(\mathbf{1}_n/n, 1/n) \} \\
        &= \{ Q \diag(\lambda) Q^T : Q \in \mathcal{O}_n, \lambda \in \simp_n \} = S_n.
        \end{aligned}
    \]

    \item[(2)] By the definition of the spectral ball and 
    \eqref{Eq: definition_d}, we have
    %Lemma~\ref{Lemma: Simplex ball}(1),
    \[
    \begin{aligned}
        S_n \cap S(X, d) &= \{ Q \diag(\lambda) Q^T : Q \in \mathcal{O}_n, \lambda \in \simp_n \cap \simp(\bflambda(X), d) \} \\
        &= \{ Q \diag(\lambda) Q^T : Q \in \mathcal{O}_n, \lambda \in \simp(\bar{\bflambda}, \bar{d}) \},
    \end{aligned}
    \]
    where
    \[
    \left\{
    \begin{array}{l}
        \bar{d} = \frac{1}{n} \sum_{i=1}^n \min\{ d, \lambda_i(X) \}, 
        \\ [0.4ex]
        \bar{\bflambda} = \max\{ \bflambda(X), d \mathbf{1}_n \} + (\bar{d} - d) \mathbf{1}_n.
    \end{array}
    \right.
    \]
    Since $\sum_{i=1}^n \lambda_i(X) = 1$, we have
    \[
    \begin{aligned}
    	\overline{d} = & \frac{1}{n} \sum_{i=1}^n \min\{ d, \lambda_i(X) \} \\
    	& = \frac{1}{n} \left( r d + \sum_{i=r+1}^n \lambda_i(X) \right)
    	= \frac{1}{n} \left( 1 + r d - \sum_{i=1}^r \lambda_i(X) \right) ,\\
%    \end{aligned}
%    \]   
%    \[
%    \begin{aligned}
	\overline{\bflambda} & =\max\{ \bflambda(X), d \mathbf{1}_n \} + (\bar{d} - d) \mathbf{1}_n
	\\ 
	&= \overline{d} \mathbf{1}_n + (\lambda_1(X) - d, \dots, \lambda_r(X) - d, 0, \dots, 0)
	 = \bflambda(\overline{X}).
%	
%        \bflambda(\widehat{X}) &= \widehat{d} \mathbf{1}_n + (\lambda_1(X) - d, \dots, \lambda_r(X) - d, 0, \dots, 0)' \\
%        &= \max\{ \bflambda(X), d \mathbf{1}_n \} + (\bar{d} - d) \mathbf{1}_n = \bar{\bflambda}.
    \end{aligned}
    \]
In the above, we used the simple fact that $\overline{X}$ is a diagonal matrix.
    Thus, $S_n \cap S(X, d) = S(\overline{X}, \overline{d})$.
    
    \item[(3)] 
    By the definition of the spectral ball, any matrix $Y \in S(X,d)$ can be represented through its eigenvalue decomposition, which allows us to rewrite the problem as
    \begin{align}
        &\min_{Y \in S(X, d)} \langle G, Y \rangle \nonumber \\
        &= \min_{\bflambda \in \simp(\bflambda(X), d)} \min_{P \in \mathcal{O}_n} \tr(G P \diag(\bflambda) P^T) \nonumber \\
        &= \min_{\bflambda \in \simp(\bflambda(X), d)} \min_{P \in \mathcal{O}_n} \tr(P^T G P \diag(\bflambda)) \nonumber \\
        &\stackrel{\eqref{SimplexBall-New}}{=} \min_{\bflambda \in \simp_n} \min_{P \in \mathcal{O}_n}
        \tr\big(P^T G P \diag(\bflambda(X) - d \bfone_n + n d \bflambda)\big) \nonumber \\
        &= \min_{\bflambda \in \simp_n} \min_{P \in \mathcal{O}_n}
        \tr\Big(P^T G P \diag(\bflambda(X) - d \bfone_n)\Big)
        + nd \cdot \tr\Big(P^T G P \diag(\bflambda)\Big)  \nonumber \\
        &= \min_{P \in \mathcal{O}_n} \min_{\bflambda \in \simp_n}
        \tr\Big(P^T G P \diag(\bflambda(X) - d \bfone_n)\Big)
        + nd \cdot \tr\Big(P^T G P \diag(\bflambda)\Big)
         \nonumber \\
        &= \min_{P \in \mathcal{O}_n}
        \left\{
        \tr\Big(P^T G P \diag(\bflambda(X) - d \bfone_n)\Big)
        + nd \cdot 
        \min_{\bflambda \in \simp_n} 
        	\tr\Big(P^T G P \diag(\bflambda)\Big)
        \right\} .  \nonumber %  \label{Eq-GY}
%        \underbrace{\min_{\bflambda \in \simp_n} 
%        \tr\Big(P^T G P \diag(\bflambda)\Big)}_{= nd \cdot \lambda_n(G)  }
%        \right\}
    \end{align}
% We note that given optimal $P$ and $\bflambda$ in the above optimization problem, the optimal $Y^*$ is given by
%\be \label{Optimal-Y*}
% Y^* = P \diag(\bflambda(X) - d \bfone_n + n d \bflambda) P^T.
%\ee 
Due to von Neumann’s trace inequality, the second term in the last equation becomes
\be \label{Neumann-1}
\min_{\bflambda \in \simp_n} \min_{P \in \mathcal{O}_n}
\tr\big(P^T G P \diag(\bflambda)\big)
=  \lambda_n(G).
\ee
Therefore
\begin{align} \label{Neumann-2}
	&\min_{Y \in S(X, d)} \langle G, Y \rangle  \nonumber \\
	&= nd \cdot \lambda_n(G) + 
	  \min_{P \in \mathcal{O}_n}
	  \left\{
	  \tr\Big(P^T G P \diag(\bflambda(X) - d \bfone_n)\Big) \right\}
	   \nonumber \\
	&= nd \cdot \lambda_n(G) +
	\langle \bflambda_{\uparrow}(G), \bflambda(X) - d \bfone_n \rangle,	
\end{align}
where last identity follows again the von Neumann’s trace inequality. 
%  %  where the last identity follows from~\eqref{SimplexBall-New}.
%    Next, we decompose the objective into two parts and bound them separately:
%    \begin{align*}
%        &\min_{\bflambda \in \simp_n} \min_{P \in \mathcal{O}_n}
%        \tr\big(P^T G P \diag(\bflambda(X) - d \bfone_n)\big)
%        + nd \cdot \tr\big(P^T G P \diag(\bflambda)\big) \\
%        \leq\;&
%        \min_{P \in \mathcal{O}_n}
%        \tr\big(P^T G P \diag(\bflambda(X) - d \bfone_n)\big)
%        + \min_{\bflambda \in \simp_n} \min_{P \in \mathcal{O}_n}
%        nd \cdot \tr\big(P^T G P \diag(\bflambda)\big).
%    \end{align*}
%    For the first term, applying von Neumann’s trace inequality yields
%    \[
%        \min_{P \in \mathcal{O}_n}
%        \tr\big(P^T G P \diag(\bflambda(X) - d \bfone_n)\big)
%        = \langle \bflambda_{\uparrow}(G), \bflambda(X) - d \bfone_n \rangle.
%    \]
%    For the second term,applying again von Neumann’s trace inequality, we obtain
%    
%    Combining the two parts, we conclude that
%    \[
%        \min_{Y \in S(X, d)} \langle G, Y \rangle
%        = \langle \bflambda_{\uparrow}(G), \bflambda(X) - d \bfone_n \rangle
%        + nd \lambda_n(G).
%    \]
%    
Moreover, both minima in \eqref{Neumann-1} and \eqref{Neumann-2} are attained when $P^T G P = \diag(\bflambda_{\uparrow}(G))$ and $\bflambda = \bfe_1$. 
Substituting these optimal choices into the parameterization in $Y$
  yields the desired solution:
    \[
        Y^* = P \diag(\bflambda(X) - d \mathbf{1}_n + n d \mathbf{e}_1) P^T.
    \]
    
    \item[(4)] Since $\| X - Y \|_F \leq d$, by the Hoffman-Wielandt theorem,
    \[
    \| \bflambda(X) - \bflambda(Y) \| \leq \| X - Y \|_F \leq d.
    \]
    By Lemma~\ref{Lemma: Simplex ball}(3), we have $\bflambda(Y) \in \simp(\bflambda(X), d)$. Thus, by the definition of the spectral ball, $Y \in S(X, d)$.
\end{enumerate}
\end{proof}

%We make some useful observations. 
When constructing the spectral ball $S(\overline{X}, \overline{d})$ in
Lemma~\ref{Lemma: Spectral-ball}(2), it is only meaningful to 
set $d < \lambda_1(X)$ (i.e., $d$ is strictly less than the largest 
eigenvalue of $X$). Otherwise, $r$ is not well defined.
Let us see what the set $S(\overline{X}, \overline{d})$ is
when $X = X_1$ and $d < \lambda_1(X_1) = 1$ from Example~\ref{Example-Nonconvexity}.
It is easy to see that $r=1$,
\[
  \overline{d} = \frac 12 (1+d-1) = \frac 12 d
  \quad
  \mbox{and} \quad
 \overline{X} = \diag(1- d/2, d/2) ,
\]
%Therefore,
\begin{align*}
	\Delta( \bflambda (\overline{X}), \overline{d} )
	&= \left\{
	 \begin{pmatrix}
	 	1- d \\
	 	0
	 \end{pmatrix} + d \bflambda \ \Big| \ 
	 \bflambda \in \Delta_2
	\right\} 
	= \left\{
	 \begin{pmatrix}
	 	1-\beta \\
	 	\beta 
	 \end{pmatrix} \ \Big| \ 0 \le \beta \le d
	\right\} .
\end{align*}
Therefore,
\[
 S(\overline{X}, \overline{d})=
 \left\{
 Q \begin{pmatrix}
 	1- \beta & \\
 	& \beta 
 \end{pmatrix} Q \ \Big|\ \ Q \in \calO_2, \ 0 \le \beta \le d  
 \right\} \\
\]
Comparing with the set $S(X_1, d)$, we see the set 
$S(\overline{X}, \overline{d})$ retains all the positive semidefinite matrices in $S(X_1, d)$. In other words, it is the set of the projections of the set 
$S(X_1, d)$ to the positive semidefinite cone $\calS^2_+$.
Once again $S(\overline{X}, \overline{d})$ is not convex. But a linear optimization has a closed-form solution that can be cheaply computed.

%%%-----------------------------------
\subsection{Spectral oracle}
\label{subsection-spectral-oracle}

We now formally define the proposed linear minimization oracle based on the spectral ball, referred to as the \emph{spectral oracle} (SO). 
The spectral oracle can be regarded as a natural extension of the standard LMO.
Its constraint set is a spectral ball rather than a spectrahedron.  

\begin{definition}[Spectral oracle]\label{Def: SO}
    Given a linear objective $G\in\mathbb{S}^n$, radius $d>0$ and a point $X\in S_n$, a solution $Y^*\in \SO(X,d,G)$ is referred to as a spectral oracle if
    \begin{equation}\label{Eq: SO}
  %      Y^*\in \argmin_{Y\in S_n \cap S(X,d)}\langle Y,G\rangle.
  Y^*\in \argmin \left\{ \langle Y,G\rangle \ | \
     Y\in S_n \cap S(X,d)   \right\} .
    \end{equation}
\end{definition}

\begin{definition}[Thresholding rank of a matrix] \label{Def-ThresholdingRank}
Let $X \in \calS^n$ and $d>0$ be given, the thresholding rank of $X$ with respect to $d$ is
\[
 r_d(X) :=  \max \{ i \in [n] \mid \lambda_i(X) > d \}.
% \| \bflambda(  X - d I_n )_+ \|_0 .
%  \quad \mbox{where} \quad 
% \calT_d(X) := \bflambda(X) \circ \sgn \Big(
% ( \bflambda(X) - d \bfone_n )_+ 
% \Big) .
\]
\end{definition} 
The thresholding rank $r_d(X)$ counts the number of eigenvalues of $X$ bigger than the thresholding
value $d$.
It has been used in Lemma~\ref{Lemma: Spectral-ball}(2).
%It has $r_d(X) = \max \{ i \in [n] \mid \lambda_i(X) > d \}$, which is used in Lemma~\ref{Lemma: Spectral-ball}(2).
%For a given $d>0$, let us define the corresponding eigenvalue thresholding operator $\calT: \calS^n \mapsto \mathbb{R}^n$ by
%\[
% \calT_d(X) := \bflambda(X) \circ \sgn \Big(
%  ( \bflambda(X) - d \bfone_n )_+ 
% \Big) 
% \quad \mbox{and} \quad
% r(X) := \| \calT_d(X) \|_0 ,
%\]
%where $\| \bfx\|_0$ is the zero-norm of $\bfx$ counting the nonzero elements 
%in $\bfx$.
%In fact, the thresholding operator $\calT_d(X)$ is the eigenvalue vector of the projection matrix 
%of $(X- d I_n)$ onto $\calS^n_+$.
%Obviously, $\calT_d(X)$ only retains those eigenvalues of $X$ bigger than the given threshold $d$ and  
%$r(X)$ counts how many such eigenvalues there are.
%Furthermore, the spectral Simplex ball $S_n \cap S(X, d) = S(\widehat{X}, \widehat{d})$ is given
%by
%\[
%  \widehat{d} = \frac 1n \Big(
%   (d+1)r_d(X) - \langle \bfone_n, \calT_d(X) \rangle 
%  \Big), \ 
%  \widehat{X} = \widehat{d} I_n 
%  + \diag\Big( (\lambda(X) - d\bfone_n)_+ \Big).
%\]
%It follows from Lemma~\ref{Lemma: Spectral-ball}(2) that the optimal
%solution of \eqref{Eq: SO} is given by
%\[
%  Y^* = V^\top \diag( \calT_d(X) ) V + n \widehat{d} \bfv_n \bfv_n^\top,
%\]
%where 
%$$ V := [\bfv_n, \cdots, \bfv_1]\ \ \mbox{satisfies} \ \
%   G = V^T \diag( \bflambda_{\uparrow} (G)) V.
%$$  
%An important observation is that, although problem~\eqref{Eq: SO} is non-convex, it admits a simple and efficient solution procedure, as summarized below.

\begin{definition}[Thresholding mapping] \label{Mapping-H}
The thresholding mapping
$\calH: \calS^n \times \calS^n \times \mathbb{R}_{+} \times [n] \to \calS^n$ is defined by
\[
 \calH(X, G, d, r) := V \Big(  \Lambda_r - d I_r \Big) V^T + n \widehat{d} \bfv_n \bfv_n^T, \quad
 X, G \in \calS^n, \ d >0 , \ r \in [n],
\]
where
\be \label{H-V}
\left\{
\begin{array}{lll}
	\Lambda_r &=&\diag\{\lambda_1(X),\dots,\lambda_r(X)\} \\ [0.6ex]
	\widehat{d} &= &\frac{1}{n}(1+rd-\sum_{i=1}^r\lambda_i(X)) \\ [0.6ex]
	V &= &[\bfv_n,\dots,\bfv_{n-r+1}]=\mbox{EV}_r(G)\in\mathbb{R}^{n\times r} .
\end{array} 
\right.
\ee 
\end{definition}

The spectral oracle defined in Def.~\ref{Def: SO} is given in the following result.

\begin{lemma}\label{Lemma-SO}
    The spectral oracle $Y^*$ of the optimal solution to Problem \eqref{Eq: SO} is given by
   \[
     Y^* = \calH(X, G, d, r_d(X)) .
   \]
\end{lemma}

\begin{proof}
For given $X \in \calS^n$ and $d>0$, we have $S_n \cap S(X, d) = S(\overline{X}, \overline{d})$,
where $(\overline{X}, \overline{d})$ are defined in \eqref{Eq-spectral-d}. 
We note that when $r$ is taken to be the thresholding rank $r_d(X)$, $\widehat{d} = \overline{d}$,
where $\widehat{d}$ is defined in \eqref{H-V}. Furthermore,
\[
  \bflambda(\overline{X}) - \overline{d} \bfone_n 
  = (\lambda_1(X ) -d, \cdots, \lambda_{r_d(X)} (X) -d, 0, \cdots, 0 ).
\]
 The solution formula in 
Lemma~\ref{Lemma: Spectral-ball}(3) yields
  \begin{eqnarray*}
        && \argmin_{Y\in S_n \cap S(X,d)}\langle Y,G\rangle 
        = \argmin_{Y\in S(\overline{X},\overline{d})}\langle Y,G\rangle \\
        &=& [\bfv_n,\dots,\bfv_1]\diag(\bflambda(\overline{X})-\overline{d}\bfone_n
        + n\overline{d}\bfe_1)[\bfv_n,\dots,\bfv_1]^T \\
%        &=& [\bfv_n,\dots,\bfv_1]\diag\{\lambda_1(X)-d+n\widehat{d},\dots,\lambda_r(X)-d,0,\dots,0\}[\bfv_n,\dots,\bfv_1]^T \\
        &=& V\diag\{\lambda_r(X)-d,\dots,\lambda_1(X)-d\}V^T+n\widehat{d}\bfv_n\bfv_n^T \\
        &=& V\Lambda_{r}V^T-dVV^T+n\widehat{d}\bfv_n\bfv_n^T\\
        &=&  \calH(X, G, d, r_d(X)).
    \end{eqnarray*}
where we note once again $r= r_d(X)$. 
%    where the first identity uses Lemma~\ref{Lemma: Spectral-ball}(2) and the second uses Lemma~\ref{Lemma: Spectral-ball}(3).
%    This matches the output of Algorithm~\ref{Alg: SO}.
\end{proof}

The spectral oracle solution $Y^*$ requires determining $r_d(X)$, which counts the number of positive eigenvalues of $(X - d I_n)$ and it is a classical matrix signature problem.
%\[
%            \lambda_1(X)\geq \cdots \geq \lambda_r(X)>d\geq \lambda_{r+1}(X)\geq\cdots \lambda_n(X).
%\]
The standard approach employs an $LDL$ decomposition~\cite{ashcraft1998accurate} of $(X - d I_n)$ and counts the positive diagonal entries of $D$ according to Sylvester’s law of inertia, which incurs a computational cost of $O(n^3)$. 
Once $r_d(X)$ is determined, the computation of $Y^*$ needs one more eigen-decomposition on the
matrix $G$ requiring $O(r^2n^2)$ cost. 
A routine extension is to employ  the new spectral oracle in a standard FW framework to get a 
linearly convergent FW method. 
A major reason against such a straightforward extension is that the Spectral FW
\cite{ding2020spectral} has already achieved such linear rate target if the oracle cost is allowed
to be more than $O(r^2n^2)$ with $r\ge r^*$ (the optimal rank).
Our main task is to develop a new procedure that brings down the oracle complexity below
$O((r^*)^2n^2)$ and only requires $r \le r^*$ for each iteration.

%\section{Main results}
\section{Spectral simplex FW methods} 
\label{section-main-results}

After the preparation above, we present our main algorithmic contributions in this section. 
We begin with a new Efficient Spectral Oracle (ECO) based on a new concept of efficient rank
and develop a natural Frank--Wolfe variant based on ECO, which serves as a baseline and reveals a key rank-adaptive property. 
We then build upon this insight to develop a refined variant that achieves linear convergence while preserving rank adaptivity.

Let us formally state the subproblem encountered in the spectral Simplex FW methods 
for the problem \eqref{Eq: Problem} at
the current iterate $(X_{t-1}, d_{t-1})$: % with the Simplex ball radius $d_{t-1}$:
\be \label{Subproblem}
Y_t \in \argmin_{Y}\ \langle Y, \nabla f(X_{t-1}) \rangle
\quad \mbox{s.t.} \ Y \in S(X_{t-1}, d_{t-1}).
\ee 
By Lemma~\ref{Lemma-SO}, the optimal solution is given by
\be \label{Yt}
{Y}_t = \calH(X_{t-1}, \nabla f(X_{t-1}), d_{t-1}, r_t)
= V (\Lambda_{ r_t} - d_{t-1} I_{r_t} )V^T + n \widehat{d} \bfv_n \bfv_n^T  ,
\ee 
where
\[ %\be \label{Yt-V}
\left\{
\begin{array}{ll}
	r_t &= r_{d_{t-1}}(X_{t-1}) \\ [0.6ex]
	\Lambda_{ r_t} &= \diag( \lambda_1 ( X_{t-1} ), \cdots, 
	\lambda_{r_t } ( X_{t-1} )) \\ [0.6ex]
	V &= 
	\begin{bmatrix}
		\bfv_n, & \cdots, & \bfv_{ n - {r}_t +1}
	\end{bmatrix}
	= \mbox{EV}_{{r}_t}  ( \nabla f( X_{t-1} ) ) \\ [0.6ex]
	\widehat{d} &= \frac 1n \Big( 
	1 + {r}_t d_{t-1} - \sum_{i=1}^{{r}_t } \lambda_i ( X_{t-1} )
	\Big) .
\end{array} 
\right .
\] %\ee

%%%%%%
\subsection{Efficient ranks}
The thresholding rank $r_t$ plays an important role in computing $Y_{t-1}$ and
obtaining it 
would compute all eigenvalues of $X_{t-1}$ at cost $O(n^3)$, which is incompatible with matching the standard Frank--Wolfe per-iteration cost. 
The main purpose of this part is to develop an alternative called efficient rank.

%{\bf (A) Defining efficient rank} for $(X_{t-1}, d_{t-1})$.
%Suppose we are at $(X_{t-1}, d_{t-1})$.
The proposed efficient rank for the iterate $X_{t-1}$ depends on the efficient rank information of the previous iterates. 
Let
\[
\widehat{r}_0 = 0 \quad \mbox{and define}\quad
\overline{r}_t := \max\left\{
\widehat{r}_i \ | \ i=0, 1, \ldots, t
\right\}.
\]
At $(X_{t-1}, d_{t-1})$, compute only $( \overline{r}_{t-1} +1 )$ largest
eigenvalues of $X_{t-1}$.
Update $\widehat{r}_t$ by
\be \label{Update-rhat}
\widehat{r}_t =
\left\{ 
\begin{array}{ll}
r_t = \max \{i\in [\bar{r}_{t-1}]\mid \lambda_i(X_{t-1})>d_{t-1}\} 
& \mbox{if} \ d_{t-1}\geq \lambda_{\overline{r}_{t-1}+1}(X_{t-1}) \\
\overline{r}_{t-1} +1 & \mbox{otherwise} ,
\end{array} 
\right .
\ee 
where we adopt the convention $\max \emptyset = 0$.
For easy reference, we call $\widehat{r}_t$ the efficient rank at
$(X_{t-1}, d_{t-1})$ because it depends on both of them.
Associated with $\widehat{r}_t$, we compute
\be \label{Yhat}
\widehat{Y}_t = \calH(X_{t-1}, \nabla f(X_{t-1}), d_{t-1}, \widehat{r}_t)
= V_t ( \Lambda_{ \widehat{r}_t}  - d_{t-1} I_{ \widehat{r}_t } ) V^T_t 
    + n \widehat{d}_{t-1} \bfv_n \bfv_n^T  ,
\ee 
where
\be \label{Yhat-V}
\left\{
\begin{array}{ll}
\Lambda_{ \widehat{r}_t} &= \diag( \lambda_1 ( X_{t-1} ), \cdots, 
\lambda_{\widehat{r}_t } ( X_{t-1} )) \\ [0.3ex]
V_t &= 
\begin{bmatrix}
	\bfv_n, & \cdots, & \bfv_{ n - \widehat{r}_t +1}
\end{bmatrix}
= \mbox{EV}_{\widehat{r}_t}  ( \nabla f( X_{t-1} ) ) \\
\widehat{d}_{t-1} &= \frac 1n \Big( 
1 + \widehat{r}_t d_{t-1} - \sum_{i=1}^{\widehat{r}_t } \lambda_i ( X_{t-1} )
\Big) .
\end{array} 
\right .
\ee 
This matrix  $\widehat{Y}_t$ will replace our theoretical $Y_t$ and
will be sufficient in achieving the desired convergence properties.
In fact, when the condition $d_{t-1}\geq \lambda_{\overline{r}_{t-1}+1}(X_{t-1})$ holds, we
have $\widehat{r}_t = r_t$ and hence $\widehat{Y}_t = Y_t$.
The following result shows that $\widehat{Y}_t$ inherits the key optimality property of $Y_t$.

\begin{lemma} \label{Lemma-Yhat}
Let $Y_t$ and $\widehat{Y}_t$ be respectively defined by \eqref{Yt} and \eqref{Yhat}.
Suppose $X_{t-1} \in S_n$.	We have
\be \label{Yhat-Sn}
 \widehat{Y}_t \in \calC_{\widehat{r}_t} ( \nabla f(X_{t-1}) ) \subseteq S_n ,
\ee 
with $\calC_r(\cdot)$ being the spectral $r$th set introduced in Lemma~\ref{Lemma-Ding},
and
\begin{eqnarray}
\langle \nabla f(X_{t-1}), \widehat{Y}_t \rangle 
&\le& \langle \nabla f(X_{t-1}), Y_t \rangle  \label{Optimality-Condition} \\
&\le& \langle \nabla f(X_{t-1}), X \rangle , \quad \forall  \ X \in S_n \cap S(X_{t-1}, d_{t-1})
   \nonumber  .
\end{eqnarray} 
\end{lemma} 

\begin{proof}
For the case $d_{t-1}\geq \lambda_{\overline{r}_{t-1}+1}(X_{t-1})$, we have $\widehat{r}_t = r_t$
and hence $\widehat{Y}_t = Y_t$, which automatically satisfies
 $\widehat{Y}_t \in S_n \cap S(X_{t-1}, d_{t-1}) \cap \calC_{r_t}( \nabla f(X_{t-1}) ) \subseteq S_n$. 
Otherwise, we have $d_{t-1} < \lambda_{\overline{r}_{t-1}+1}(X_{t-1})$ and $\widehat{r}_t = \overline{r}_{t-1} +1$. Hence,
\[
  \lambda_i( X_{t-1}) \ge \lambda_{\widehat{r}_t  }(X_{t-1}) 
  = \lambda_{\overline{r}_{t-1}+1}(X_{t-1}) > d_{t-1}, \ \forall 
  \ i =1, \ldots, \widehat{r}_t.
\]
Consequently, the eigenvalues of $\widehat{Y}_t$ are nonnegative (i.e., $\widehat{Y}_t \succeq 0$)
and
\[
 \tr( \widehat{Y}_t ) = \sum_{i=1}^{ \widehat{r}_t } ( \lambda_i (X_{t-1}) - d_{t-1}  )
 + \Big( 1 + \widehat{r}_t d_{t-1} - \sum_{i=1}^{ \widehat{r}_t }  \lambda_i (X_{t-1}) \Big)
 = 1 .
\]
As in $Y_t$, the eigenvectors of $\widehat{Y}_{t}$ are the ones corresponding to the $\widehat{r}_t$
smallest eigenvalues of $\nabla f(X_{t-1})$.
Hence, we proved $\widehat{Y}_t \in \calC_{\widehat{r}_t}( \nabla f(X_{t-1}) ) \subseteq  S_n$.

We now prove \eqref{Optimality-Condition}. Once again, if $\widehat{Y}_t = Y_t$, then the 
result comes from the optimality condition of $Y_t$ being the optimal solution of 
subproblem \eqref{Subproblem}. We only consider the case $d_{t-1} < \lambda_{\overline{r}_{t-1}+1}(X_{t-1})$.
For simplicity, let $r := r_t$, $\widehat{r} := \widehat{r}_t$, $G :=\nabla f(X_{t-1})$ and $\lambda_i(X_{t-1}) = \lambda_i$.
It is important to note that $V$ and $V_t$ contain the orthonormal 
eigenvectors of $G$. 
We hence obtain
\[
\langle V_t, GV_t \rangle = \sum_{j=1}^{ \hat{r} } \lambda_{n-j+1} (G) ,\
\langle V, GV \rangle = \sum_{j=1}^{ {r} } \lambda_{n-j+1} (G) 
\ \mbox{and} \
G\bfv_n = \lambda_n(G) \bfv_n .
\] 

We then have
\begin{align*}
\langle G,\; Y_t - \widehat{Y}_t \rangle 
&= \langle G,\; Y_t  \rangle - \langle G,\; \widehat{Y}_t \rangle 
\label{Key-Inequality} \\
&= \langle V^T G V, \; \Lambda_{r} - d_{t-1} I_r \rangle 
-\langle V^T_t G V_t, \; \Lambda_{\widehat{r}} - d_{t-1} I_{\widehat{r}} \rangle 
\nonumber \\
&\quad  + \Big(
(r - \widehat{r}) d_{t-1} - \sum_{j=\widehat{r} +1 }^r \lambda_j
\Big) \langle G, \; \bfv_n \bfv_n^T \rangle 
\nonumber \\
&= \sum_{j=\widehat{r} +1 }^r \Big(
\lambda_j - d_{t-1}
\Big) \lambda_{n-j+1}(G) + \sum_{j=\widehat{r} +1 }^r (d_{t-1}-\lambda_j) \lambda_n(G) 
\nonumber \\
&= \sum_{j=\widehat{r} +1 }^r 
\underbrace{( \lambda_j - d_{t-1} )}_{\ge 0} 
\underbrace{( \lambda_{n-j+1}(G) - \lambda_n(G) )}_{\ge 0}
\ge 0 . \nonumber 
\end{align*}
This proves \eqref{Optimality-Condition}.
\end{proof}

\begin{remark}
We recall that the ideal situation is to compute $Y_t$, which enjoys the first-order optimality
condition and hence facilitates the convergence analysis. 
However, computing $Y_t$ would requires $r_t$, which costs $O(n^3)$.
We therefore propose its efficient rank $\widehat{r}_t$ and compute $\widehat{Y}_t$. 
This modification relaxes the feasibility requirement and only ensures $\widehat{Y}_t \in S_n$,
not necessarily in the feasible region $S_n \cap S(X_{t-1}, d_{t-1})$ of the subproblem 
\eqref{Subproblem}. It still satisfies the optimality condition that $Y_t$ enjoys. 
The reward for this relaxation is that the cost of computing it is reduced below $O((r^*)^2 n^2)$ as
we show below.
\end{remark} 

\subsection{Natural FW variant with ERSO}
\label{subsubsection-fast-rt}

In this subsection, we consider a natural baseline approach: directly incorporating the Efficient Rank Spectral Oracle (ERSO) $\widehat{Y}_t$ proposed in the previous section into the Frank--Wolfe framework by replacing the standard linear minimization oracle. 
This yields a new algorithm, as shown in Alg.~\ref{Alg:SO-rFW}, which we refer to as \textit{ERSO-FW}, and we investigate its theoretical properties.
%Unlike the polytope-constrained setting, where replacing the standard oracle with the simplex oracle leads to linear convergence \cite{wang2025simplex}, it turns out that SO-FW does not immediately enjoy such an improvement. 
%Nevertheless, 
We show that it achieves the same $O(1/\epsilon)$ convergence rate as the standard Frank--Wolfe method (see Theorem~~\ref{Thm-sublinear-SOFW}).

More interestingly, despite the lack of global linear convergence, the algorithm exhibits an inherent \emph{rank-adaptive} property (see Theorem~\ref{Thm-rank-adaptive}), which is not present in the classical spectral Frank--Wolfe method. 
In addition, when embedded into the Frank--Wolfe iterations, the Spectral Oracle admits a more efficient implementation, leading to an overall computational cost of $O(r^* n^2)$ per iteration.

\begin{algorithm}[htbp]
	\footnotesize
	\renewcommand{\algorithmicrequire}{\textbf{Input:}}
	\renewcommand{\algorithmicensure}{\textbf{Output:}}
	\caption{(ERSO-FW) Frank--Wolfe with Efficient-Rank Spectral Oracle}
	\label{Alg:SO-rFW}
	\begin{algorithmic}[1]
		\REQUIRE $X_0\in \spec_n$, initial lower bound $B_0\leq f^*$, step-sizes $\{\eta_t\}_{t\geq 1}\subseteq [0,1]$
		\STATE Set $d_0\gets\sqrt{\frac{2(f(X_0)-B_0)}{\mu}}$ and $\widehat{r}_0 = 0$.
		\FOR{$t=1,\dots$}
		\STATE Compute $\widehat{r}_t$ by \eqref{Update-rhat} and
		$\widehat{Y}_t = \calH(X_{t-1}, \nabla f(X_{t-1}), d_{t-1}, \widehat{r}_t)$ by \eqref{Yhat}.
		\STATE Working lower bound: $B_t^w\gets f(X_{t-1})+\langle\nabla f(X_{t-1}), \widehat{Y}_t-X_{t-1}\rangle$.
		\STATE $B_t\gets \max\{B_{t-1}, B_t^w \}$.
		\STATE $X_t\gets (1-\eta_t)X_{t-1}+\eta_t \widehat{Y}_t$ for some $\eta_t\in [0,1]$.
		\STATE $d_t\gets \sqrt{\frac{2(f(X_t)-B_t)}{\mu}}$.
		\ENDFOR
	\end{algorithmic}
\end{algorithm}

\begin{theorem}\label{Thm-sublinear-SOFW}
Using Algorithm \ref{Alg:SO-rFW} with step-size $\eta_t=\frac{2}{t+1}$ or the exact line-search, we have $X_t \in S_n$ and 
\[
    f(X_t)-f^*\leq f(X_t)-B_t\leq \frac{4L}{t+1}, \quad \forall \ t \ge 1 .
\]
\end{theorem}

\begin{proof}
Since we start with $X_0 \in S_n$, Lemma~\ref{Lemma-Yhat} implies $\widehat{Y}_t \in S_n$ and hence
$X_t = (1 - \eta_t)X_{t-1} + \eta_t \widehat{Y}_t \in S_n$ for all $t\ge 1$ by induction.
	
We now claim that $S(X_t,d_t)\cap \Xstar\neq\emptyset$ and $B_t\leq f^*$. We prove this by induction.  First, we have 
    \begin{equation*}
        (\mu d_0^2)/2=f(X_0)-B_0\geq f(X_0)-f^*\geq \mu\dist^2(X_0,\Xstar)/2.
    \end{equation*}
    This implies that $\dist(X_0,\Xstar)\leq d_0$, and by Lemma \ref{Lemma: Spectral-ball}(4), we have $\Xstar\cap S(X_0,d_0)\neq\emptyset$. 
    Therefore, the claim holds for $t=0$. 
    
    Now suppose $S(X_{t-1},d_{t-1})\cap\Xstar\neq\emptyset$ and $B_{t-1}\leq f^*$.
    Picking any $X^*$ in this intersection, we obtain
    \begin{eqnarray}
        B_t^w
        &=& f(X_{t-1})+\langle\nabla f(X_{t-1}), \widehat{Y}_t-X_{t-1}\rangle \nonumber \\
        &\stackrel{\eqref{Optimality-Condition}}{=}& f(X_{t-1})+\langle\nabla f(X_{t-1}),Y_t-X_{t-1}\rangle 
          \nonumber \\
        &\leq&  f(X_{t-1})+\langle\nabla f(X_{t-1}),X^*-X_{t-1}\rangle
         \label{Important-Inequality} \\
        &\leq& f(X^*)=f^*, \nonumber
    \end{eqnarray}
    where the last inequality uses convexity of $f$.
    Thus $B_t=\max\{B_{t-1},B_t^w\}\leq f^*$.
    As in the case $t=0$, we have
    $\frac{\mu}{2}d_t^2\geq\frac{\mu}{2}\dist^2(X_t,\Xstar)$,
    and hence $\Xstar\cap S(X_t,d_t)\neq\emptyset$. 
    
    We now start to prove $f(X_t)-B_t\leq \frac{4L}{t+1}$. Let $\alpha_t = 2/(t+1)$, then for both the exact 
    line-search step or $\eta_t=2/(t+1)$, we have
    \begin{align*} \label{InEq-fxt}
       & f(X_t)
        \leq  f(X_{t-1}+\alpha_t ( \widehat{Y}_t-X_{t-1})) \\
        &\leq  f(X_{t-1})+\frac{2}{t+1} \langle \nabla f(X_{t-1}), \widehat{Y}_t-X_{t-1}\rangle+\frac{2L}{(t+1)^2}\lVert \widehat{Y}_t-X_{t-1} \rVert^2 
        \nonumber \\
        &\le  \left(
        1 - \frac 2{t+1}
        \right) f(X_{t-1})
        + \frac{2}{t+1}  \Big( f(X_{t-1}) +
        \langle \nabla f(X_{t-1}), \widehat{Y}_t-X_{t-1}\rangle
        \Big) +\frac{4L}{(t+1)^2}    \nonumber \\
        &=\frac{t-1}{t+1}f(X_{t-1})+\frac{2}{t+1}B_t^w+\frac{4L}{(t+1)^2} 
         \nonumber \\
        &\le \frac{t-1}{t+1}f(X_{t-1})+\frac{2}{t+1}B_t +\frac{4L}{(t+1)^2}
        \qquad (\mbox{because}\ B_t \ge B_t^w) 
        \nonumber \\
        &= \frac{t-1}{t+1}f(X_{t-1})+ B_t - \frac{t-1}{t+1}B_t +\frac{4L}{(t+1)^2}
         \nonumber  \\
        &\le \frac{t-1}{t+1}f(X_{t-1})+ B_t - \frac{t-1}{t+1}B_{t-1} +\frac{4L}{(t+1)^2} \qquad (\mbox{because}\ B_t \ge B_{t-1}) \nonumber
    \end{align*}
    where the third inequality used the fact that the diameter of $S_n$ is 
    $\sqrt{2}$ and $X_{t-1}$, $\widehat{Y}_t \in S_n$. 
% Since $B_t=\max\{B_{t-1}, B_t^w \}$, we have $f(X_t)-B_t\leq \frac{t-1}{t+1}(f(X_{t-1}-B_{t-1})+\frac{4L}{(t+1)^2}$ and thus
Therefore,
    \[
    \begin{aligned}
        t(t+1)(f(X_t)-B_t)&\leq t(t-1)(f(X_{t-1}-B_{t-1})+\frac{4tL}{t+1} \\
        &\leq t(t-1)(f(X_{t-1}-B_{t-1})+4L.
    \end{aligned}
    \]
    Summing up from $i=1$ to $t$, we have
    \[
        t(t+1)(f(X_t)-B_t)\leq 4tL,
    \]
    which concludes that $f(X_t)-B_t\leq \frac{4L}{t+1}$.
\end{proof}

The following theorem establishes the \emph{rank-adaptive property} of the proposed algorithm. 
It shows that throughout the iterations, the number of eigenvalues computed at each step never exceeds $r^*$, 
and eventually stabilizes at $r^*$ after a finite number of iterations. 
This property is particularly desirable in practice, as the true rank $r^*$ is often unknown a priori. 
Our algorithm therefore does not require any prior estimate of $r^*$ as a hyperparameter, while automatically adapting to it during the optimization process.

% 下面的定理说明了，我们迭代过程中需要计算的特征值个数，始终不会超过r^*，并且当迭代不断进行，最终会始终等于r^*。这是非常好的事情，因为在实际问题中，r^*的值往往是未知的，我们的算法并不需要对r^*的先验来确定超参数，同时还能在运行过程中得到对r^*的估计。

\begin{theorem}\label{Thm-rank-adaptive}
Concerning the efficient rank sequence $\{ \widehat{r}_t\}$ and the thresholding rank sequence 
$\{r_t\}$, we have 
(i) $\widehat{r}_t \le r_t \le r^*$ for all $t \ge 0$;
(ii) $r_t = r^*$ for all $t \geq T_0:= \frac{72L}{\mu \lambda_{r^*}^2}$; and
(iii) $\widehat{r}_t = r_t= r^*$ for all $t  \geq T_0' := T_0 + r^*$.
\end{theorem}

\begin{proof}
    We first establish that $r_t \leq r^*$ for all $t \in \mathbb{N}$. Let $X^* = \argmin_{Y \in \Xstar} \lVert Y - X_{t-1} \rVert$. Then,
    \begin{align} \label{Bound-dt}
        d_{t-1}&=\sqrt{\frac{2(f(X_{t-1})-B_{t-1})}{\mu}} 
        \ge \sqrt{\frac{2(f(X_{t-1})-f^*)}{\mu}} \\
        &\geq \dist(X_{t-1},\Xstar) = \lVert X_{t-1}-X^*\rVert_F \nonumber \\
        &\geq \lvert \lambda_{r^*+1}(X_{t-1}) - \lambda_{r^*+1}(X^*) \rvert \nonumber \\
        &= \lambda_{r^*+1}(X_{t-1}), \nonumber
    \end{align}
where the first inequality used the fact $B_t$ is a lower bound for $f^*$,
     the second inequality is due to the QG condition and 
     the third is Weyl's inequality.
  %  By definition of $r_t$,
  The thresholding rank $r_t$ means that
 \be\label{Def-rt}
        \lambda_1(X_{t-1}) \geq \cdots \geq \lambda_{r_t}(X_{t-1}) > d_{t-1} \geq \lambda_{r_t+1}(X_{t-1}) \geq \cdots \lambda_n(X_{t-1}).
 \ee
 This, together with \eqref{Bound-dt}, implies $r_t < r^*+1$. Consequently, $r_t\leq r^*$ 
 for all $t\ge 0$. Now we recall the definition of the efficient rank $\widehat{r}_t$. 
 For the case $d_{t-1} \ge \lambda_{\overline{r}_{t-1} +1 }(X_{t-1})$, $\widehat{r}_t = r_t$. 
 Hence $\widehat{r}_t \le r_t$ for this case. We consider the other case where 
  $d_{t-1} < \lambda_{\overline{r}_{t-1} +1 }(X_{t-1})$. Comparing with the condition
\eqref{Def-rt}, we must have $\overline{r}_{t-1} + 1 \le r_t$, i.e., 
$\widehat{r}_t = \overline{r}_{t-1} +1 \le r_t$. In both cases, we proved
$\widehat{r}_t \le r_t \le r^*$.
    
Next, we prove parts (ii) and (iii) of Theorem~\ref{Thm-rank-adaptive}
in two steps. First, we prove that $r_t = r^*$ for $t \ge T_0$.
%For $t \geq T_0$  %\frac{72L}{\mu \lambda_{r^*}^2}$,
Thm.~\ref{Thm-sublinear-SOFW} implies
    \begin{equation}\label{Eq-proof-rank-adapative-1}
        d_{t-1}=\sqrt{\frac{2}{\mu}(f(X_{t-1})-B_{t-1})}\leq \sqrt{\frac{2}{\mu}\cdot\frac{4L}{t}}\leq \frac{\lambda_{r^*}}{3}.
    \end{equation}
It follows from \eqref{Bound-dt} that
%    Following a similar argument in the first part, we obtain 
    \begin{equation}\label{Eq-proof-rank-adapative-2}
        d_{t-1} \ge \| X_{t-1} - X^*\|_F
        \geq \lvert\lambda_{r^*}(X_{t-1})-\lambda_{r^*}(X^*) \rvert = \lvert\lambda_{r^*}(X_{t-1})-\lambda_{r^*} \rvert.
    \end{equation}
    Combining \eqref{Eq-proof-rank-adapative-1} with \eqref{Eq-proof-rank-adapative-2}, we deduce
    \[
        \lambda_{r^*}(X_{t-1})
        \ge \lambda_{r^*} - d_{t-1} 
        \ge \lambda_{r^*} - \frac 13 \lambda_{r^*}
        = \frac{2}{3}\lambda_{r^*}>\frac{\lambda_{r^*}}{3}\geq d_{t-1}.
    \]
% Once again, by \eqref{Def-rt}
%    \[
%        \lambda_1(X_{t-1})\geq \cdots \geq \lambda_{r_t}(X_{t-1})>d_{t-1}\geq \lambda_{r_t+1}(X_{t-1})\geq\cdots \lambda_n(X_{t-1}),
%    \]
 Once again, it follow from \eqref{Def-rt}
  that $r_t \geq r^*$. Since $r_t \leq r^*$ from the first part, we conclude that $r_t = r^*$ for all $t \geq T_0$. %\frac{72L}{\mu \lambda_{r^*}^2}$.
  
The second step is to prove for $t \ge T_0' = T_0 + r^*$, we will have $\widehat{r}_t = r^*$.
Let $t_0 \ge T_0$. We have $r_{t_0}=r^*$ from the first part. 
Suppose at $X_{t_0-1}$, the condition $d_{t_0-1}\geq \lambda_{\overline{r}_{t_0-1}+1}(X_{t_0-1})$ holds.
Then by the update rule of $\widehat{r}_t$, we have $\widehat{r}_{t_0} =r_{t_0} = r^*$. Consequently,
$\overline{r}_{t_0} = \max\{\widehat{r}_0, \cdots, \widehat{r}_{t_0}\} =r^*$.
Since $\{\overline{r}_t \}$ is non-decreasing, we must have $\overline{r}_t = \overline{r}_{t_0} = r^*$ for all $t > t_0$. Furthermore, $\widehat{r}_t \ge \min\{r_t, \overline{r}_{t-1}\} = r^*$ for
all $t> t_0$. Since we have proved $\widehat{r}_t \le r^*$ for all $t$, we have
$\widehat{r}_t = r^*$ for all $t\ge t_0$. What we just proved is the following: if the condition $d_{t_0-1}\geq \lambda_{\overline{r}_{t_0-1}+1}(X_{t_0-1})$ holds once at some $t_0 \ge T_0$, then $\widehat{r}_t = r^*$ for all $t \ge t_0$. The remaining case is that the condition does 
not hold at all within the range $T_0 \le t \le T_0'$. But the update rule of $\widehat{r}_t$ would increase the rank bound $\overline{r}_t$ by one each time this condition does not hold. This means
that it cannot happen more than $r^*$ times because $\widehat{r}_t \le r^*$ for all $t$. 
Combining the two cases we just considered proves that $\widehat{r}_t = r^*$ for all
$t \ge T_0'$. This proves our claim.
\end{proof}

\begin{remark} \label{Remark-EfficientRank}
(i) When computing the efficient rank $\widehat{r}_t$ by
\eqref{Update-rhat}, 
it requires computing only the largest $(\overline{r}_{t-1} + 1)$ eigenvalues of $X_{t-1}$ once. 
Since $\overline{r}_t = \max\{\widehat{r}_0, \dots, \widehat{r}_t\} \leq r^*$ by Thm.~\ref{Thm-rank-adaptive}, 
the overall computational complexity of computing $\widehat{r}_t$ is $O(r^* n^2)$, which is comparable to that of the standard Frank–Wolfe oracle.
(ii) Although $\overline{r}_t$ is non-decreasing, $\widehat{r}_t$ is not.
It may decrease or increase depending whether the condition
$d_{t-1}\geq \lambda_{\overline{r}_{t-1}+1}(X_{t-1})$ occurs or not.
If it happens, $\widehat{r}_t = r_t$ and the thresholding rank $r_t$ may decrease. 
Our numerical experiments confirm this behaviour.
\end{remark}
%\paragraph{Computational cost of Algorithm~\ref{Alg:SO-rFW}.}
%In each iteration, Algorithm~\ref{Alg: compute_rt} requires computing only the largest $(\bar{r}_{t-1} + 1)$ eigenvalues of $X_{t-1}$ once. 
%Since Theorem~\ref{Thm-rank-adaptive} relies solely on Theorem~\ref{Thm-sublinear-SOFW}, we still have $\bar{r}_t = \max\{r_1, \dots, r_t\} \leq r^*$. 
%Consequently, the overall computational complexity of Algorithm~\ref{Alg: compute_rt} is $O(r^* n^2)$, which is comparable to that of the standard Frank–Wolfe oracle. 

%%%%----------------------------------
%\subsection{Linearly convergent and rank-adaptive spectral oracle variant}
\subsection{Linearly convergent spectral Simplex-FW (S-SFW)}
% 在这一subsection，我们希望在继承上一个算法的rank adaptive 性质的基础上，进一步改进使得算法可以线性收敛。改进的基本想法是借鉴Frank-Wolfe的fully-Corrective variant \cite{jaggi2013revisiting} 和\cite{ding2020spectral}提出的Spectral Frank-Wolfe,当通过Spectral oracle计算出秩r_t后，求解一个r_t维的子问题，算法细节如Alg.~\ref{Alg: SO-FW-FC}所示。
In this subsection, we further enhance the proposed rank-adaptive framework to achieve \emph{linear convergence}. 
The main idea builds upon the rank-adaptive ERSO-FW algorithm introduced in the previous subsection, while incorporating the fully-corrective strategy of the Frank–Wolfe method~\cite{jaggi2013revisiting,ding2020spectral}. 

Specifically, after obtaining the efficient rank $\widehat{r}_t$ through the Spectral Oracle, we solve a low-dimensional subproblem of size $\widehat{r}_t$ to refine the update direction. 
This local correction step effectively reduces the optimization residual within the subspace spanned by the current eigenvectors, leading to improved convergence behavior without sacrificing the algorithm’s rank adaptivity. 
The detailed procedure of the proposed linearly convergent variant is summarized in Algorithm~\ref{Alg: SO-FW-FC}.

\begin{algorithm}[htbp]
\footnotesize
\renewcommand{\algorithmicrequire}{\textbf{Input:}}
\renewcommand{\algorithmicensure}{\textbf{Output:}}
\caption{(S-SFW) Rank-adaptive linearly convergent Frank--Wolfe with Spectral Oracle}
\label{Alg: SO-FW-FC}
\begin{algorithmic}[1]
    \REQUIRE $X_0\in \spec_n$, initial lower bound $B_0\leq f^*$.
    \STATE Set $d_0\gets\sqrt{\frac{2(f(X_0)-B_0)}{\mu}}$ and $\widehat{r}_0 = 0$.
    \FOR{$t=1,\dots$}
        \STATE Compute $\widehat{r}_t$ by \eqref{Update-rhat} and
        $\widehat{Y}_t = \calH(X_{t-1}, \nabla f(X_{t-1}), d_{t-1}, \widehat{r}_t)$ by \eqref{Yhat} and $V_t$ in \eqref{Yhat-V}.
        \STATE Working lower bound: $B_t^w\gets f(X_{t-1})+\langle\nabla f(X_{t-1}), \widehat{Y}_t-X_{t-1}\rangle$.
        \STATE $B_t\gets \max\{B_{t-1}, B_t^w \}$.
        \STATE Solve $\min_{S\in \spec_{\widehat{r}_t},\eta\in [0,1]} f((1-\eta)X_{t-1}+\eta V_tSV_t^T)$ to obtain $(S_t, \eta_t)$.
        \STATE $X_t\gets (1-\eta_t)X_{t-1}+\eta_tV_tS_tV_t^T$.
        \STATE $d_t\gets \sqrt{\frac{2(f(X_t)-B_t)}{\mu}}$.
    \ENDFOR
\end{algorithmic}
\end{algorithm}

We now state the theoretical guarantees for Algorithm~\ref{Alg: SO-FW-FC}. 

\begin{theorem} \label{Thm-SSFW-linear}
Using Algorithm~\ref{Alg: SO-FW-FC},  we have for all $t$,
\[
    f(X_t)-f^*\leq \frac{4L}{t+1}.
\]
Moreover, it holds that $r_t = r^*$ for all $t \geq T_0$ and $\widehat{r}_t = r^*$ for all $t \ge T_0'$.
Define $T_1:=\frac{72L^3}{\mu\delta^2}$, then for $t \geq \max\{T_0', T_1\}$, we further have
\[
     f(X_t) - f^* 
    \leq \left( 1 - \min \left\{ \frac{\mu}{8L}, \frac{\delta}{12L} \right \} \right)\, 
    (f(X_{t-1}) - f^*).
\]
\end{theorem}

\begin{proof}
For simplicity, let $\calC_t := \calC_{\widehat{r}_t} ( \nabla f(X_{t-1}) )$. 
By \eqref{Yhat-Sn}, $\widehat{Y}_t \in \calC_t \subseteq S_n$.
The optimality condition on $(S_t, \eta_t)$ implies
\begin{eqnarray} 
	f(X_t) &=& f( (1-\eta_t) X_{t-1} + \eta_t V_t S_t V_t^T  )  \label{fxt}  \\
	&=& f(X_{t-1}  + \eta_t ( V_t S_t V_t^T - X_{t-1}  ) ) \nonumber \\
	&\le& f(X_{t-1}  + \eta ( W - X_{t-1}  ) ) \qquad \forall \ 
	\eta \in [0,1], \ W \in \calC_t . \nonumber 
\end{eqnarray}
Since $\widehat{Y}_t \in \calC_t$, we have
\[
 f(X_t) \le f(X_{t-1}  + \alpha_t (  \widehat{Y}_t- X_{t-1}  ) ).
\]
This recovers the first inequality below \eqref{Important-Inequality} and we can follow from there in the proof of Thm.~\ref{Thm-sublinear-SOFW} to establish the first bound. Once this bound is valid, the proof of Thm.~\ref{Thm-rank-adaptive} goes through line-by-line and we can establish that
$r_t = r^*$ for $t\ge T_0$ and $\widehat{r}_t = r^*$ for all $t \ge T_0'$.

%    \begin{equation}\label{Eq-proof-SOFW-linear-1}
%        f(X_t) \leq f(X_{t-1})+\eta\langle W-X_{t-1},\nabla f(X_{t-1}) \rangle + \frac{L\eta^2}{2}\|W-X_{t-1} \|_F^2.
%    \end{equation}
%    % hold for both Option 1 and 2.
%    Since $Y_t \in \mathcal{C}_t$, following the same argument as in Theorems~\ref{Thm-sublinear-SOFW} and~\ref{Thm-rank-adaptive}, we obtain that
%    \[
%        f(X_t)-f^* \leq \frac{4L}{t+1}, \quad \forall\, t \geq 1.
%    \]
%    Furthermore, when $t \geq T_0 := \tfrac{72L}{\mu \lambda_{r^*}^2}$, we have $r_t = r^*$. 
%   
Once we have the sublinear bound proved above, the rest is just a straightforward extension of
the proof techniques developed in \cite{ding2020spectral}. For completeness, we include
the details below. 
    Let $X^* = \argmin_{Y \in \Xstar} \lVert Y - X_{t-1} \rVert$. Then, for all $t> T_1$, we have
    \begin{align*}
        & \quad  \lVert \nabla f(X_{t-1})-\nabla f(X^*)\rVert_F 
        \leq L\lVert X_{t-1}-X^*\rVert_F \\
        &\leq L\left(\frac{2(f(X_{t-1})-f^*)}{\mu} \right)^{\frac{1}{2}} 
        \leq \sqrt{\frac{8L^3}{\mu t}}\leq\frac{\delta}{3},
    \end{align*}
    where we used $L$-smoothness, the QG condition and the bound just established.
    Hence, for $t \geq \max\{T_0', T_1\}$, 
    \begin{align*}
        &\lambda_{n-\widehat{r}_t}(\nabla f(X_{t-1})) - \lambda_{n-\widehat{r}_t+1}(\nabla f(X_{t-1})) \\
        =&\lambda_{n-r^*}(\nabla f(X_{t-1})) - \lambda_{n-r^*+1}(\nabla f(X_{t-1})) \\
        =& \underbrace{\lambda_{n-r^*}(\nabla f(X^*)) - \lambda_{n-r^*+1}(\nabla f(X^*))}_{=\delta} 
        + \underbrace{(\lambda_{n-r^*}(\nabla f(X_{t-1})) - \lambda_{n-r^*}(\nabla f(X^*)))}_{\geq -\frac{1}{3}\delta} \\
        &+ \underbrace{\lambda_{n-r^*+1}(\nabla f(X^*)) - \lambda_{n-r^*+1}(\nabla f(X_{t-1}))}_{\geq -\frac{1}{3}\delta} \\
        \geq & \frac{1}{3}\delta,
    \end{align*}
    where the lower bounds on the last two terms follow from Weyl's inequality.
    From \eqref{fxt}, it follows that for any $\eta \in [0,1]$ and $W \in \calC_t$,
    \begin{align*}
        & f(X_t)-f^* \le f( X_{t-1} + \eta ( W - X_{t-1})  ) \\
        \leq & (f(X_{t-1})-f^*)+\eta \langle W-X_{t-1},\nabla f(X_{t-1})\rangle 
        % &\qquad 
        + \tfrac{L\eta^2}{2}\lVert W-X_{t-1} \rVert_F^2 \\
        \leq & (f(X_{t-1})-f^*) + \eta \langle X^*-X_{t-1},\nabla f(X_{t-1})\rangle\\
        &\ + \eta \langle W-X^*,\nabla f(X_{t-1})\rangle 
         + L\eta^2\bigl(\lVert X^*-W \rVert_F^2 + \lVert X_{t-1}-X^* \rVert_F^2\bigr) \\
        \leq & (1-\eta)(f(X_{t-1})-f^*)+\eta \langle W-X^*,\nabla f(X_{t-1})\rangle \\
        &\ + L\eta^2\bigl(\lVert X^*-W \rVert_F^2 + \lVert X_{t-1}-X^* \rVert_F^2\bigr) \\
        \leq & \Bigl(1-\eta + \frac{2L\eta^2}{\mu}\Bigr)(f(X_{t-1})-f^*) 
        +\eta \langle W-X^*,\nabla f(X_{t-1})\rangle + L\eta^2\lVert X^*-W \rVert_F^2.
    \end{align*}
    Here we used $(a+b)^2\le 2a^2+2b^2$, convexity of $f$, and the QG condition.
Note that the above bound holds for any $W \in \calC_t$ and Lemma~\ref{Lemma-Ding} holds for 
$Y=\nabla f(X_{t-1})$. There exists $W_t \in \mathcal{C}_t$ such that
    \[
        \langle W_t-X^*,\nabla f(X_{t-1}) \rangle \leq -\frac{\delta}{6}\lVert X^*-W_t \rVert_F^2.
    \]
    Combining these results gives
    \begin{align*}
        & f(X_t)-f^* 
        \leq \Big(1-\eta + \frac{2L\eta^2}{\mu} \Big)(f(X_{t-1})-f^*) 
        + \Big( L\eta^2-\frac{\delta\eta}{6} \Big) \lVert X^*-W_t \rVert_F^2.
    \end{align*}
    It suffices to choose $\eta$ such that 
    \[
        1 - \eta + \tfrac{2L \eta^2}{\mu} < 1 
        \quad \text{and} \quad
        L \eta^2 - \tfrac{\delta \eta}{6} \le  0.
    \]
They are equivalent to
\[
  \mu > 2 L\eta \quad \mbox{and} \quad \delta \ge 6L \eta.
\]
The following choices satisfy those conditions.
    When $\mu \leq \tfrac{2}{3}\delta$, setting $\eta = \tfrac{\mu}{4L}$ yields
    \[
        f(X_t) - f^* \leq \big(1 - \tfrac{\mu}{8L}\big)(f(X_{t-1}) - f^*).
    \]
    When $\mu > \tfrac{2}{3}\delta$, setting $\eta = \tfrac{\delta}{6L}$ gives
    \begin{align*}
        f(X_t) - f^* 
        &< \Big(1 - \tfrac{\delta}{6L} + \tfrac{\delta^2}{18L\mu}\Big)(f(X_{t-1}) - f^*) 
        < \Big(1 - \tfrac{\delta}{12L}\Big)(f(X_{t-1}) - f^*),
    \end{align*}
    where the last inequality uses $\mu > \tfrac{2}{3}\delta$.
 Combining the two cases yields our desired linear convergence inequality.
% 
%    Therefore, for all $t \geq \max\{T_0, T_1\}$,
%    \[
%        f(X_t) - f^* 
%        \leq \left(1 - \min\left\{\tfrac{\mu}{8L}, \tfrac{\delta}{12L} \right\}
%        \right)(f(X_{t-1}) - f^*),
%    \]
%    which completes the proof.
\end{proof}

\begin{remark} \label{Remark-SmallSDP}
Note that the subproblem in Line~6 of Algorithm~\ref{Alg: SO-FW-FC} 
is semidefinite program (SDP) in disguise as explained below.
By introducing the substitutions $\tileta = 1 - \eta$ and $\tilS = \eta S$, it can be equivalently reformulated as 
\begin{equation}\label{Eq: SOFC-FW-subproblem} 
	\min_{\tileta + \tr(\tilS) = 1,\; \tileta \ge 0,\; \tilS \succeq 0} f(\tileta X_{t-1} + V_t \tilS V_t^T). 
\end{equation} 
The above problem is a convex SDP and can be efficiently solved using projected gradient descent (PGD) or its accelerated variant (APGD). Moreover, projecting onto the feasible set 
\[ \{(\tileta, \tilS) \mid \tileta + \tr(\tilS) = 1,\, \tileta \ge 0,\, \tilS \succeq 0\} 
\] 
requires only computing the eigenvalue decomposition of an $\widehat{r}_t$-dimensional symmetric matrix and performing a projection onto the $(\widehat{r}_t + 1)$-dimensional simplex (the one more
dimension is due to the variable $\tileta$). The total computational cost of this step is therefore $O((r^*)^3)$. Detailed procedures and correctness of this projection step can be found in~\cite[Lemma~3.1]{allen2017linear} and~\cite[Lemma~6]{garber2021convergence}.
\end{remark}

%%%%-----------------------------------------------------
\section{Experiments}\label{section-numerical-tests}

In this section, we evaluate the empirical performance of the proposed algorithm on three representative low-rank optimization problems. 
We begin with \emph{quadratic sensing}, a standard benchmark in low-rank matrix recovery that has been widely adopted in prior numerical studies \cite{ding2020spectral,garber2023linear,garber2025linearly}. 
This experiment serves as the primary testbed for validating the main theoretical properties of our method, including its linear convergence behavior and rank adaptivity. 
We then consider two additional applications that have also been commonly used in related work, namely matrix completion and polynomial neural network training \cite{allen2017linear}.
Although these problems are originally formulated with nuclear norm ball constraints, they can both be equivalently reformulated as optimization problems over the spectrahedron~\cite{jaggi2010simple}. 
This unified formulation allows us to evaluate the proposed method under a common optimization framework while demonstrating its applicability beyond the quadratic sensing setting.

The remainder of this section is organized as follows. 
We first describe the shared implementation details and experimental setup. 
We then present numerical results on quadratic sensing, followed by additional experiments on matrix completion and polynomial neural networks.

\subsection{Experimental setup}

All experiments were run in MATLAB R2022b.
To avoid repeated first-order oracles when solving~\eqref{Eq: SOFC-FW-subproblem}, we minimize the $L$-smooth quadratic upper model of $f$ at $X_{t-1}$,
\begin{equation}\label{Eq: SOFC-FW-upper-subproblem}
\begin{aligned}
    \min_{\tileta + \tr(\tilS) = 1,\; \tileta \ge 0,\; \tilS \succeq 0}
    \phi (\tileta, \tilS):=& f(X_{t-1})+\langle\tileta X_{t-1}+V_t\tilS V_t^T - X_{t-1},\nabla f(X_{t-1}) \rangle \\
    +& \frac{L}{2}\lVert X_{t-1}-(\tileta X_{t-1}+V_t\tilS V_t^T) \rVert_F^2 ,
\end{aligned}
\end{equation}
using FASTA~\cite{goldstein2014field,goldstein2015fasta}; the gradient of $\phi$ is assembled from $\nabla f(X_{t-1})$ and $V_t$ only.
We compare classical FW with line search~\cite{frank1956algorithm}, Block-FW~\cite{allen2017linear} under several block sizes $k$, and the proposed S-SFW (Algorithm~\ref{Alg: SO-FW-FC}).
ERSO-FW is not reported separately, as it uses the same efficient-rank mechanism but only the sublinear guarantee of Theorem~\ref{Thm-sublinear-SOFW}.
All methods use the same $L$ and $\mu$ whenever these parameters are required.
We report the Frank--Wolfe gap $\mathrm{Gap}(X_t):=\max_{Y\in\spec_n}\langle\nabla f(X_t),X_t-Y\rangle$ (an upper bound on $f(X_t)-f^*$), the working gap $f(X_t)-B_t$ for S-SFW, and the objective gap $f(X_t)-f^*$, where $f^*$ is computed by FASTA to tolerance $10^{-10}$.

\subsection{Quadratic sensing}
In this subsection, we evaluate the empirical performance of the proposed method on synthetic instances of the quadratic sensing problem~\cite{chen2015exact}. 

\paragraph{Problem formulation}
The goal is to recover a low-rank ground truth matrix in the spectrahedron from random quadratic measurements. 
Our experimental setup closely follows those adopted in prior works~\cite{ding2020spectral,garber2023linear,garber2025linearly}, enabling a direct and fair comparison.
Specifically, we consider the following optimization problem:
\begin{equation}
    \min_{X\in\spec_n}\left\{f(X):=\frac{1}{2}\sum_{i=1}^m (\tau \bfa_i^TX\bfa_i-\bfb_i)^2 \right\}.
\end{equation}
The ground truth matrix is generated as $X_{\natural} = U_{\natural} U_{\natural}^{\top} \in \spec_n, $ where $ U_{\natural} \in \mathbb{R}^{n \times r^*} $ has i.i.d.\ standard Gaussian entries and is subsequently normalized to ensure $\tr(X_{\natural}) = \|U_{\natural}\|_F^2 = 1. $
This construction yields a rank-$r^*$ positive semidefinite matrix lying on the boundary of the spectrahedron.
The sensing vectors $\bfa_i \in \mathbb{R}^n,\ i=1,\dots,m$ are independently sampled from the standard Gaussian distribution.
The observation vector is given by $\bfb = \bfb_{\natural} + \bfb_{\text{noise}},$ where $b_{\natural}(i) = \bfa_i^{\top} X_{\natural} \bfa_i$ and the noise term is defined as $\bfb_{\text{noise}}=\frac{\|\bfb_{\natural}\|_2}{2}\,\bfv,$ with $\bfv \in \mathbb{R}^m$ being a random unit vector.
This choice yields a moderate noise level relative to the signal magnitude.
Finally, to ensure that the strict complementarity condition holds and to mitigate over-fitting effects, we fix the scaling parameter as $ \tau = 0.5, $ which is consistent with the settings used in \cite{ding2020spectral,garber2023linear,garber2025linearly}. 

In all experiments, we set the smoothness parameter to $L = 0.5\, n^2$ and the quadratic growth parameter to $\mu = 0.4\, n^2$ for all methods, following the same choices adopted in
\cite{ding2020spectral,garber2025linearly}. 

\paragraph{Rank-Adaptive Behavior and Linear Convergence}
We begin by comparing S-SFW (Algorithm~\ref{Alg: SO-FW-FC}) with the standard Frank--Wolfe method equipped with line search, while simultaneously tracking the evolution of the efficient rank $\widehat{r}_t$ along the iterations.
As shown in Figure~\ref{fig: gap_compare}, the standard Frank--Wolfe method exhibits a sublinear convergence behavior, whereas the proposed algorithm enters a clear linear convergence regime after a short burn-in phase.

Moreover, the efficient rank starts from $\widehat{r}_0 = 0$ and grows rapidly during the initial iterations.
We note in passing that $\widehat{r}_t = 0$ yields $\widehat{d}_{t-1} = 1/n$ in \eqref{Yhat-V} and hence $\widehat{Y}_t = \bfv_n \bfv_n^T$, i.e., the method automatically performs a classical rank-one Frank--Wolfe step; this is precisely what happens over the initial plateau of each rank curve in Figure~\ref{fig: gap_compare}.
In accordance with Remark~\ref{Remark-EfficientRank}(ii), $\widehat{r}_t$ is not monotone before it settles: all three instances exhibit temporary drops, and for $r^*=10$ it even returns to $0$ once and then stagnates at $9$ for about sixty iterations before reaching $10$.
In all tested instances, $\widehat{r}_t$ locks at the ground-truth rank $r^*$ after roughly $13$, $59$ and $124$ iterations for $r^* = 3, 6$ and $10$, respectively.
These numbers are one to two orders of magnitude smaller than the theoretical bound $T_0' = T_0 + r^*$ of Theorem~\ref{Thm-rank-adaptive}(iii): since $\tr(X^*)=1$ and $\mbox{rank}(X^*) = r^*$ imply $\lambda_{r^*} \le 1/r^*$, the present choice $L = 0.5\,n^2$ and $\mu = 0.4\, n^2$ gives
\[
  T_0 = \frac{72L}{\mu \lambda_{r^*}^2} = \frac{90}{\lambda_{r^*}^2} \ge 90 (r^*)^2 ,
\]
that is, $T_0' \ge 813,\ 3246$ and $9010$ for $r^* = 3, 6$ and $10$.
Throughout the iterations, $\widehat{r}_t$ never exceeds the intrinsic rank $r^*$ of the optimal solution, as guaranteed by Theorem~\ref{Thm-rank-adaptive}(i).

Notably, in all three panels the onset of linear convergence coincides with the identification of the correct rank: the FW gap begins its steep decay at the very iteration where $\widehat{r}_t$ reaches $r^*$ (marked by the dashed vertical line).
This empirical observation is fully consistent with Theorem~\ref{Thm-SSFW-linear}, which predicts linear convergence after the correct active rank has been identified.

\begin{figure}[!ht]
  \centering
  \begin{minipage}[t]{0.32\textwidth}
    \centering
    \includegraphics[width=\linewidth]{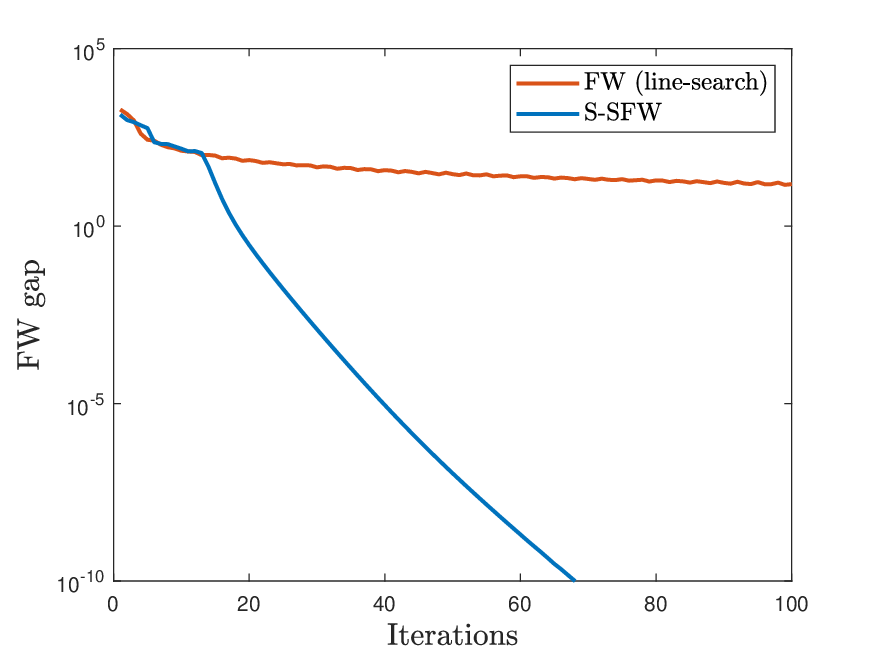}\\
    \includegraphics[width=\linewidth]{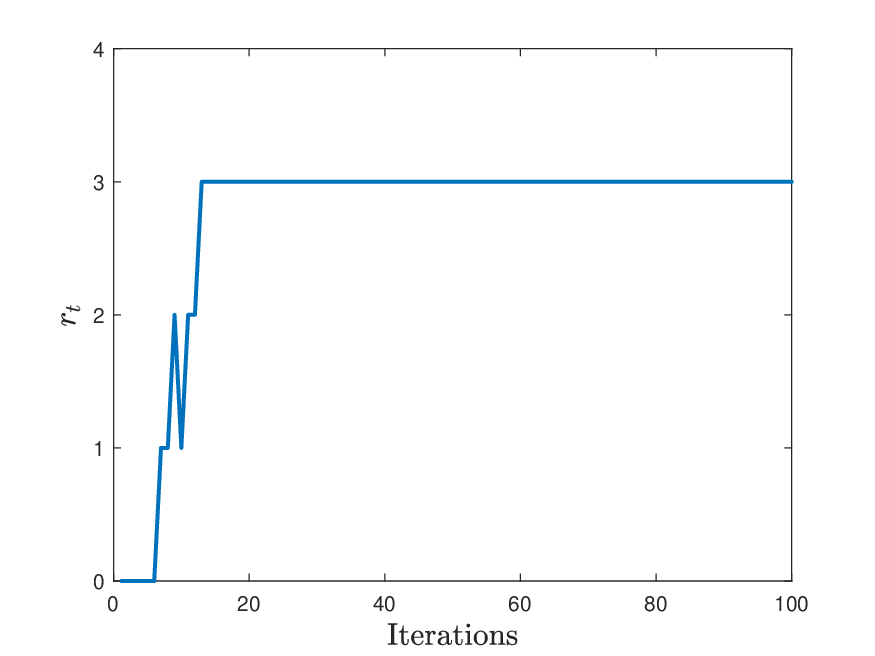}\\
    {\footnotesize (a) $r^*=3$}
  \end{minipage}\hfill
  \begin{minipage}[t]{0.32\textwidth}
    \centering
    \includegraphics[width=\linewidth]{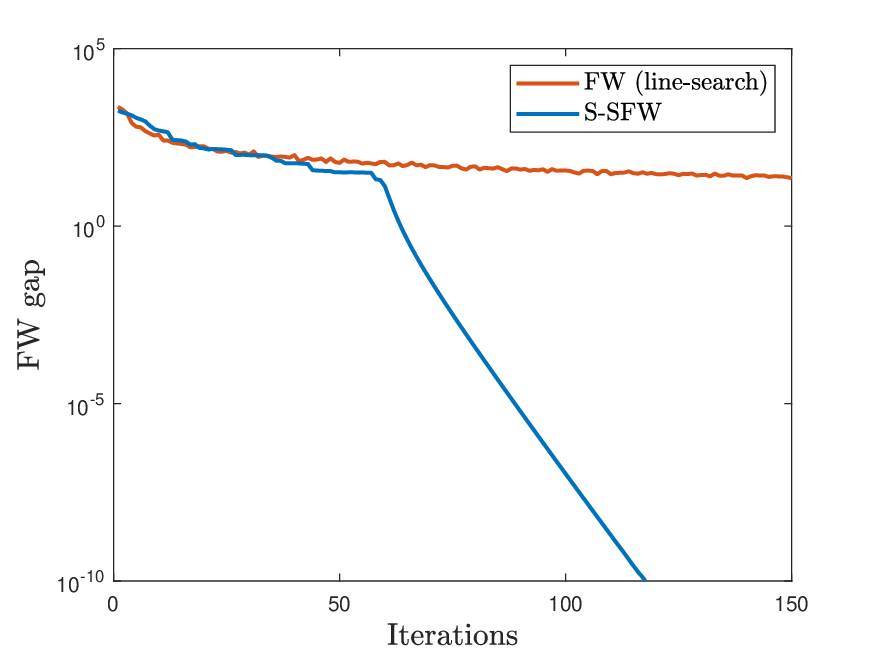}\\
    \includegraphics[width=\linewidth]{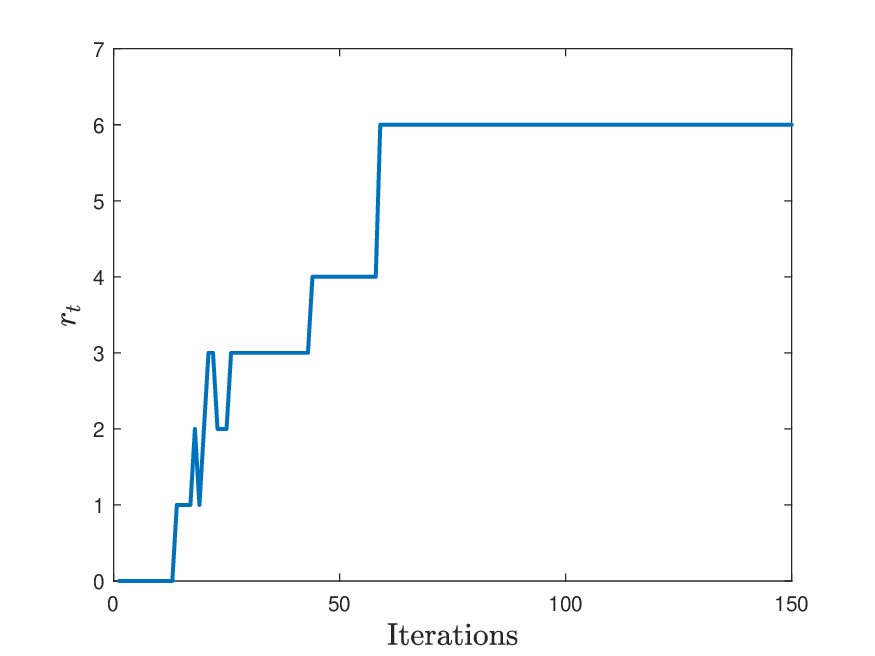}\\
    {\footnotesize (b) $r^*=6$}
  \end{minipage}\hfill
  \begin{minipage}[t]{0.32\textwidth}
    \centering
    \includegraphics[width=\linewidth]{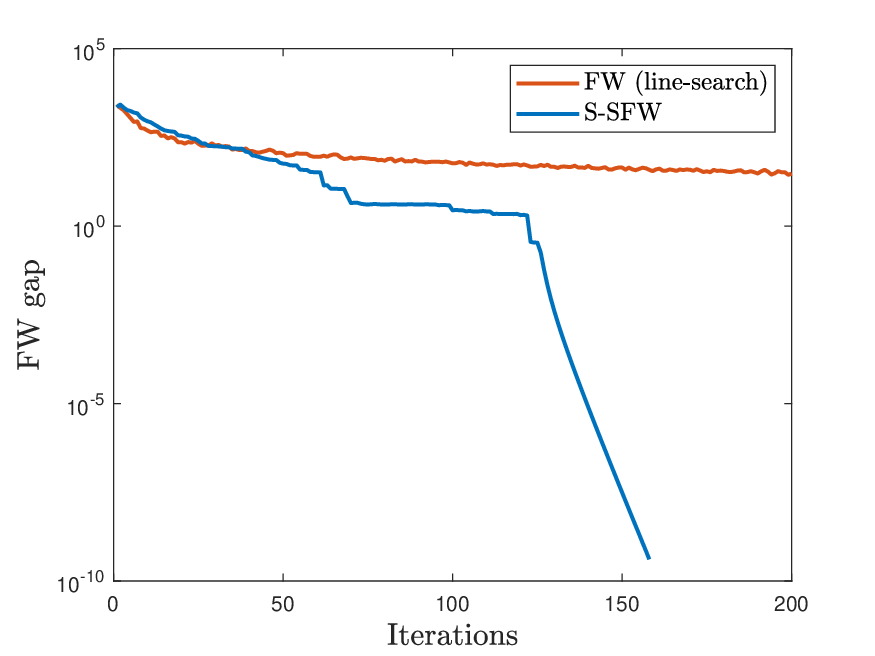}\\
    \includegraphics[width=\linewidth]{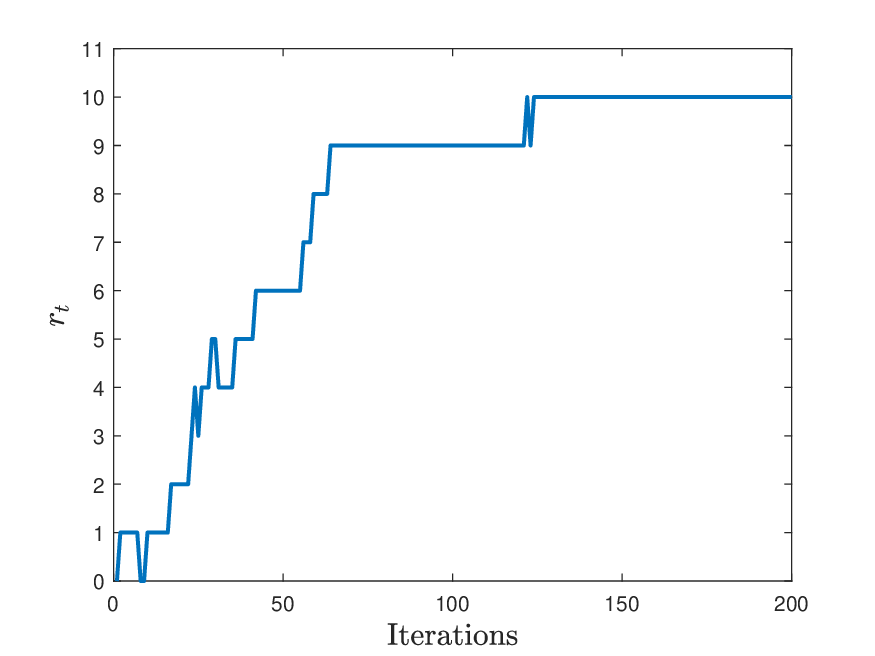}\\
    {\footnotesize (c) $r^*=10$}
  \end{minipage}
  \caption{Rank adaptivity and linear convergence for $r^*\in\{3,6,10\}$ ($n=100$, $m=15nr^*$).
  Top: FW gap of FW and S-SFW; bottom: efficient rank $\widehat{r}_t$. The dashed line marks the iteration at which $\widehat{r}_t$ locks at $r^*$.}
  \label{fig: gap_compare}
\end{figure}

\paragraph{Comparison to Frank--Wolfe Variants}
We conduct a series of numerical experiments to compare the proposed method with several Frank--Wolfe-type algorithms.
In addition to the standard Frank--Wolfe method with line search and our algorithm, we also include the Block Frank-Wolfe (Block-FW) method~\cite{allen2017linear} under two different choices of the block size parameter:
$k = r^*$ (denoted as \emph{Block-FW ($k = r^*$)} in the figures) and
$k = r^* - 1$ (denoted as \emph{Block-FW ($k = r^* - 1$)}).
For a fair comparison, the Block-FW method uses the same hyperparameters $L$ and $\mu$ as those adopted in S-SFW.
We evaluate all four algorithms on problem instances with dimensions $n\in\{100, 200, 400\}$ and ground-truth ranks $r^* \in\{ 3, 6\}$.
The corresponding results are reported in Figure~\ref{fig:fwgap_compare_grid}. 

\begin{figure}[!ht]
    \centering
    \begin{minipage}[t]{0.32\textwidth}
        \centering
        \includegraphics[width=\linewidth]{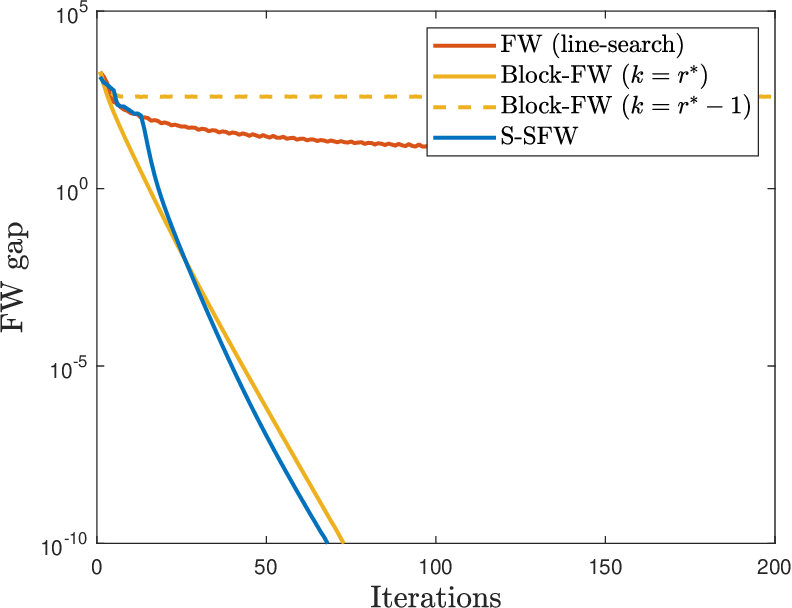}
        {\footnotesize (a) $n=100,\ r^*=3$}
    \end{minipage}\hfill
    \begin{minipage}[t]{0.32\textwidth}
        \centering
        \includegraphics[width=\linewidth]{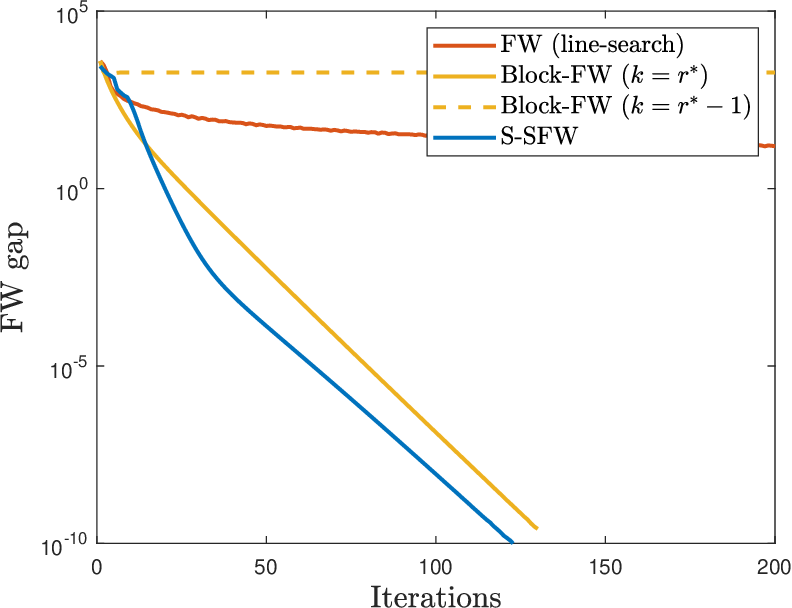}
        {\footnotesize (b) $n=200,\ r^*=3$}
    \end{minipage}\hfill
    \begin{minipage}[t]{0.32\textwidth}
        \centering
        \includegraphics[width=\linewidth]{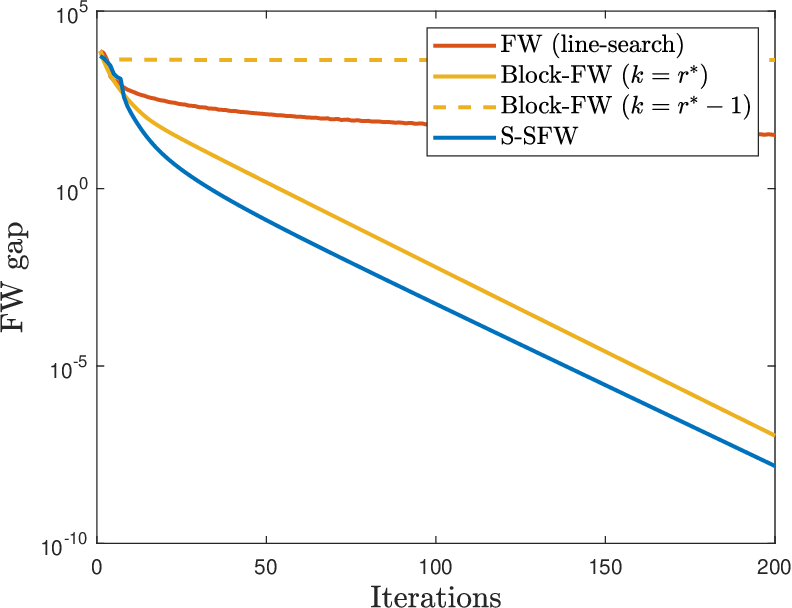}
        {\footnotesize (c) $n=400,\ r^*=3$}
    \end{minipage}

    \vspace{1em}

    \begin{minipage}[t]{0.32\textwidth}
        \centering
        \includegraphics[width=\linewidth]{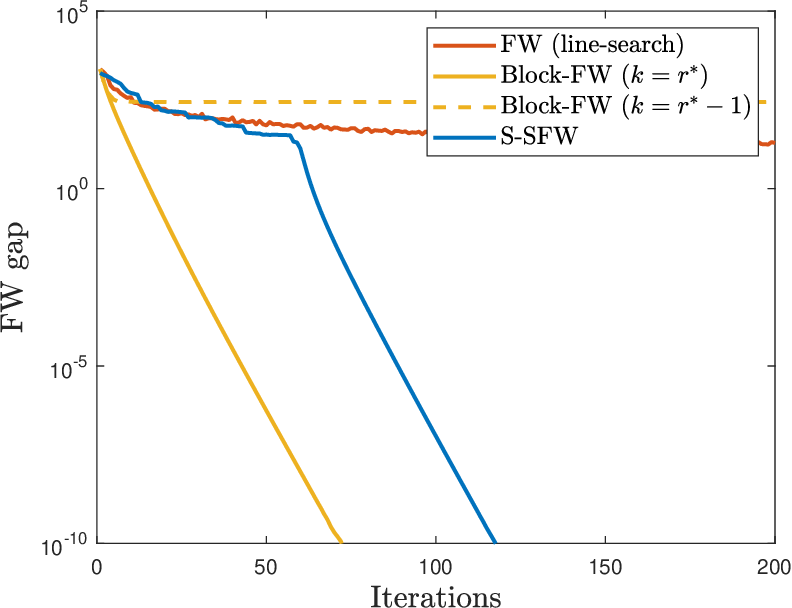}
        {\footnotesize (d) $n=100,\ r^*=6$}
    \end{minipage}\hfill
    \begin{minipage}[t]{0.32\textwidth}
        \centering
        \includegraphics[width=\linewidth]{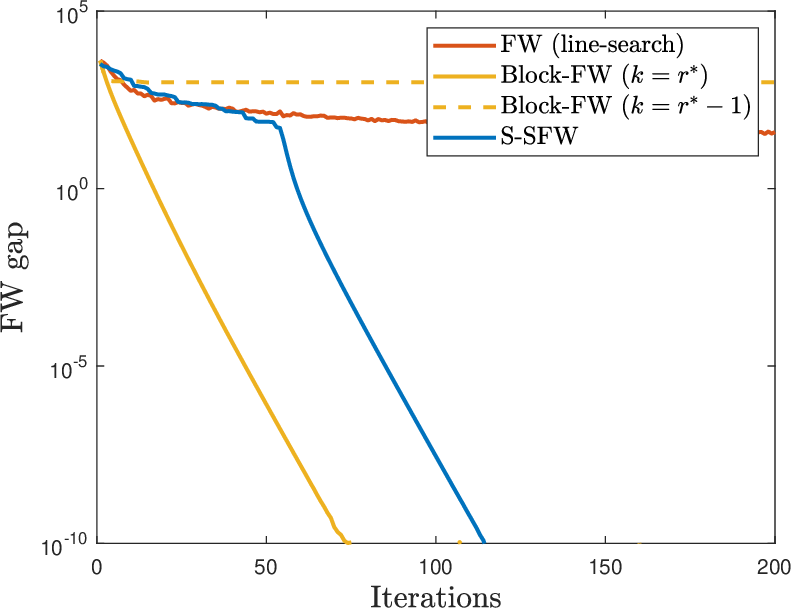}
        {\footnotesize (e) $n=200,\ r^*=6$}
    \end{minipage}\hfill
    \begin{minipage}[t]{0.32\textwidth}
        \centering
        \includegraphics[width=\linewidth]{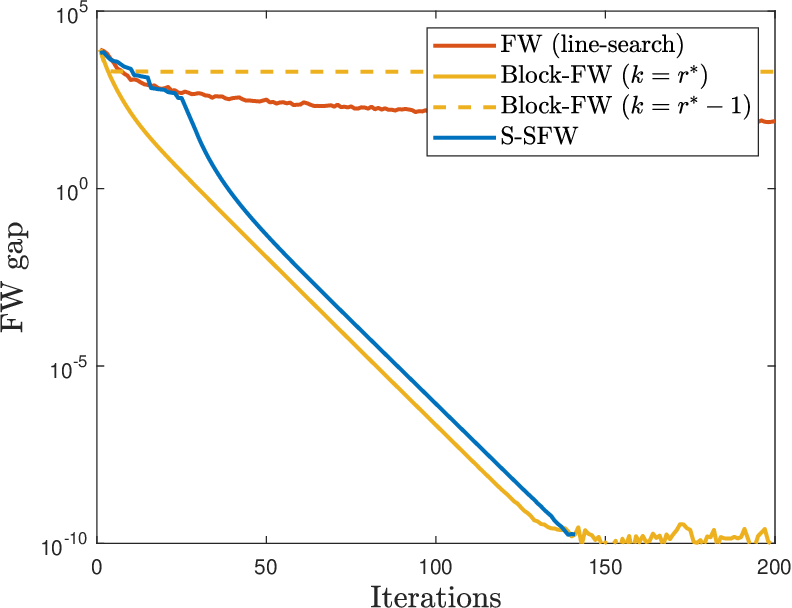}
        {\footnotesize (f) $n=400,\ r^*=6$}
    \end{minipage}
    \caption{FW gap versus iterations for different $n$ and $r^*$.}
    \label{fig:fwgap_compare_grid}
\end{figure}
Overall, the proposed algorithm demonstrates the best performance in most experimental settings.
Notably, although Block-FW with $k = r^*$ also converges rapidly, a slight misspecification of the block size (i.e., setting $k = r^* - 1$) can lead to a dramatic deterioration in performance and fails to converge in all tested instances.
This highlights the sensitivity of Block-FW to its hyperparameter choice and, in contrast, illustrates the advantage of the proposed method in automatically adapting to the intrinsic rank of the solution.

Finally, we observe that in the cases $(n, r^*) = (100, 6)$ and $(200, 6)$, the proposed method exhibits a relatively long sublinear convergence phase before entering the linear regime.
This behavior can be attributed to the use of a fixed smoothness parameter $L = 0.5 n^2$, which underestimates the actual smoothness constant of the generated instances in these settings.

\subsection{Additional applications}
We next test S-SFW on two nuclear-norm constrained problems that lift to~\eqref{Eq: Problem}: matrix completion and the training of polynomial neural networks.
\paragraph{Matrix Completion}
Suppose there is an unknown matrix $M\in \mathbb{R}^{n_1\times n_2}$ that is approximately low-rank, and only a subset $\Omega$ of its entries is observed; namely, we observe $M_{i,j}$ for every $(i,j)\in\Omega$. 
A typical example is collaborative filtering, where $M_{i,j}$ represents the rating given by user $i$ to movie $j$. 
A common convex relaxation for recovering $M$ is the following nuclear-norm constrained least-squares problem:
\begin{equation}\label{Eq-matrix-problem}
\min _{X \in \mathbb{R}^{n_1 \times n_2}}
\left\{
\left.
\frac{1}{2} \sum_{(i, j) \in \Omega}
\left(\tau X_{i, j}-M_{i, j}\right)^2
\ \right|\ 
\|X\|_* \leq 1
\right\}.
\end{equation}
Following~\cite{jaggi2010simple}, problem~\eqref{Eq-matrix-problem} is brought into the form~\eqref{Eq: Problem} by lifting $X$ to the off-diagonal block of a symmetric matrix in $\spec_{n_1+n_2}$, so that the ambient dimension of the lifted problem is $n = n_1 + n_2$.

\begin{figure}[!ht]
  \centering
  \begin{minipage}[t]{0.48\textwidth}
    \centering
    \includegraphics[width=\linewidth]{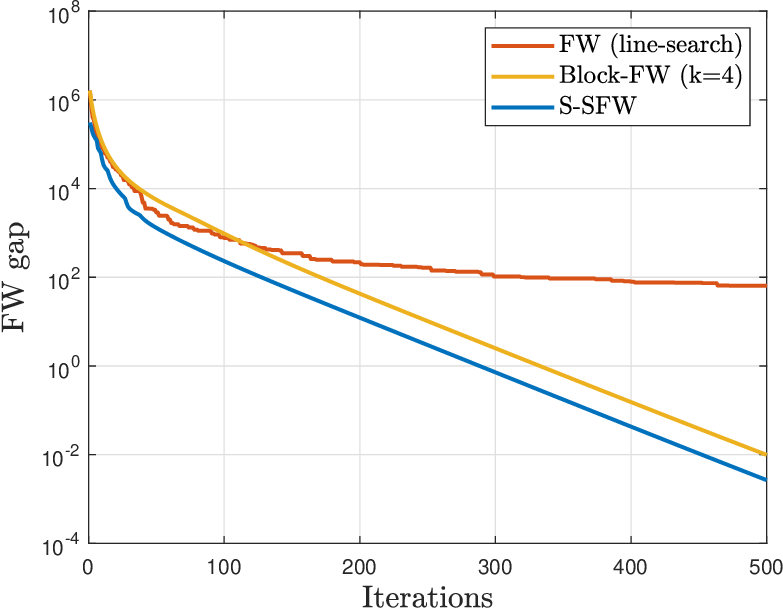}\\
    \includegraphics[width=\linewidth]{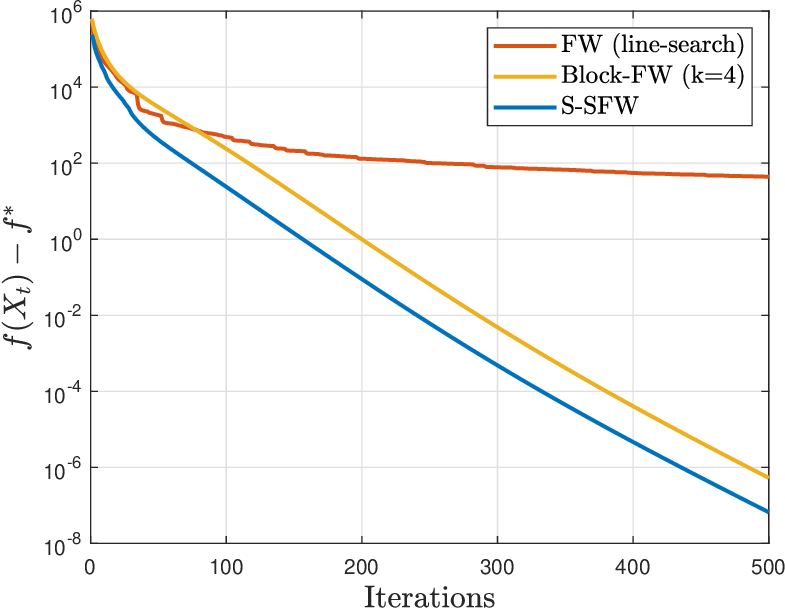}\\
    {\footnotesize (a) $\tau=2500$}
  \end{minipage}\hfill
  \begin{minipage}[t]{0.48\textwidth}
    \centering
    \includegraphics[width=\linewidth]{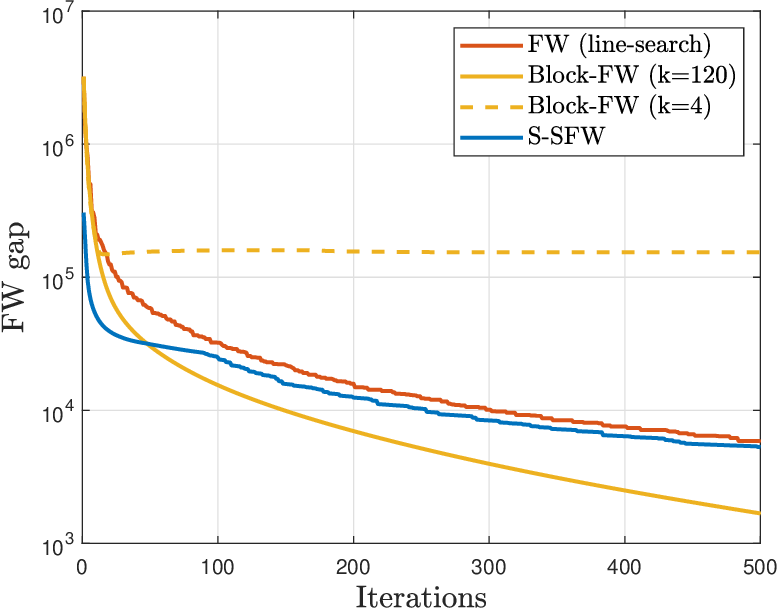}\\
    \includegraphics[width=\linewidth]{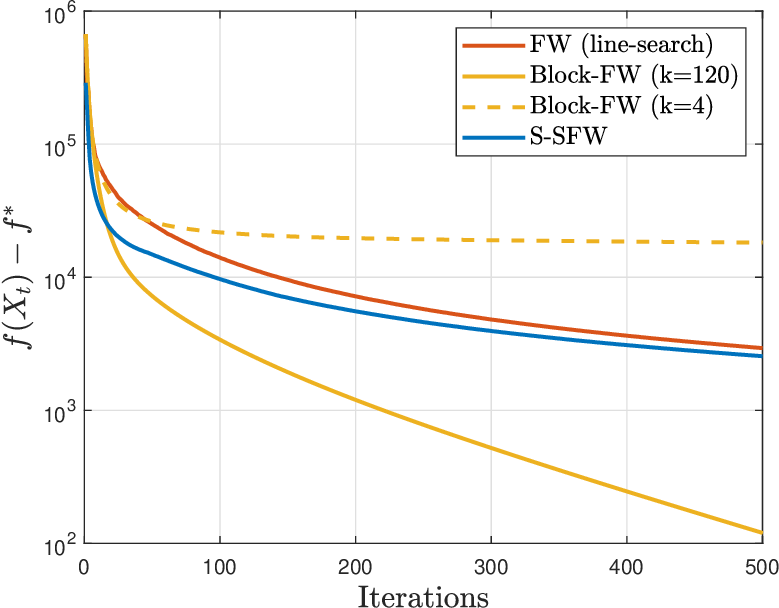}\\
    {\footnotesize (b) $\tau=5000$}
  \end{minipage}
  \caption{\textsc{MovieLens100k}: FW gap (top) and objective gap (bottom) for $\tau=2500$ (a) and $\tau=5000$ (b).}
  \label{fig: matrix_compare}
\end{figure}

We consider the same matrix completion experiment as in \cite{garber2016faster,allen2017linear} on the \textsc{MovieLens100k} dataset, where $n_1=943$, $n_2=1682$, and $|\Omega|=10^5$, so that the lifted problem has ambient dimension $n = n_1+n_2 = 2625$. 
We compare FW, Block-FW, and S-SFW for different choices of $\tau$. 
For all methods, we set $L=\tau^2$ and $\mu=0.3L$. 
For Block-FW, we choose the block size $k=4$ when $\tau=2500$, and test both $k=4$ and $k=120$ when $\tau=5000$, motivated by the fact that the optimal solution has rank $r^*=3$ for $\tau=2500$ and $r^*=117$ for $\tau=5000$.

The numerical results are reported in Fig.~\ref{fig: matrix_compare}. 
When $\tau$ is relatively small, S-SFW consistently outperforms Block-FW in terms of both the FW gap and objective optimality gap. 
When $\tau$ becomes larger, Block-FW with a sufficiently large block size ($k=120$) achieves faster convergence than S-SFW. 
However, the performance of Block-FW is highly sensitive to the choice of $k$: in particular, the variant with $k=4$ fails to converge for $\tau=5000$. 
This indicates that while Block-FW can be competitive when an appropriate block size is carefully tuned, S-SFW is substantially more robust across different parameter regimes.

\paragraph{Polynomial Neural Networks}

Following the polynomial network benchmark introduced in
\cite{allen2017linear}, we consider the training of two-layer polynomial
networks with quadratic activation $\sigma(a)=a^2$.
For input dimension $d$ and $p$ hidden neurons, the model class can be written as
\[
    P_p=
    \left\{
        \bfx\mapsto
        \sum_{j=1}^{p}
        a_j(\bfw_j^\top\bfx)^2
        \;\middle|\;
        \bfw_j\in\mathbb{R}^d,\ 
        \|\bfw_j\|_2=1,\ 
        \bfa\in\mathbb{R}^p
    \right\}.
\]

Introducing $A=\sum_{j=1}^{p}a_j\bfw_j\bfw_j^\top$, this model admits the equivalent representation
\[
    P_p=
    \left\{
        \bfx\mapsto\bfx^\top A\bfx
        \;\middle|\;
        A\in\mathbb{R}^{d\times d},\
        \mbox{rank}(A)\le p
    \right\}.
\]

Replacing the hard rank constraint by the trace-norm constraint 
$\|A\|_*\le\tau$, and rescaling $A$ by $\tau$, the problem over training data $\{(\bfx_1,y_1),\dots,(\bfx_N,y_N)\}\subset\mathbb{R}^d\times\mathbb{R}$
can be formulated as
\begin{equation}\label{Eq-neural-problem}
    \min_{\|A\|_*\le 1}
    \left\{
        \frac12
        \sum_{i=1}^{N}
        \bigl(
            \tau\,\bfx_i^\top A\bfx_i-y_i
        \bigr)^2
    \right\}.
\end{equation}

We apply model~\eqref{Eq-neural-problem} to train a binary classifier on
the \textsc{Mnist} handwritten digit dataset, distinguishing digit ``0'' from all
other digits.
The original \textsc{Mnist} training set contains $60{,}000$ images; due to memory
limitations, we use the first $N=20{,}000$ training examples in our
experiments.\footnote{This differs from \cite{allen2017linear}, which uses the full training set.}
Each image has dimension $d=28\times 28=784$, so that the same lifting as above yields an ambient dimension $n = 2d = 1568$.
We set $y_i=1$ for digit ``0'' and $y_i=0$ otherwise, and normalize pixel
values to $[0,1]$ by dividing grayscale intensities by $255$.

\begin{figure}[!ht]
  \centering
  \begin{minipage}[t]{0.48\textwidth}
    \centering
    \includegraphics[width=\linewidth]{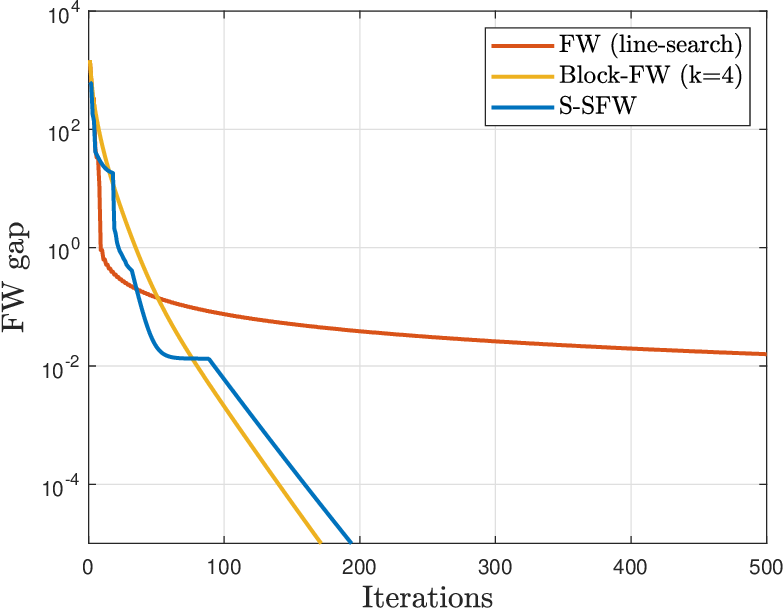}\\
    \includegraphics[width=\linewidth]{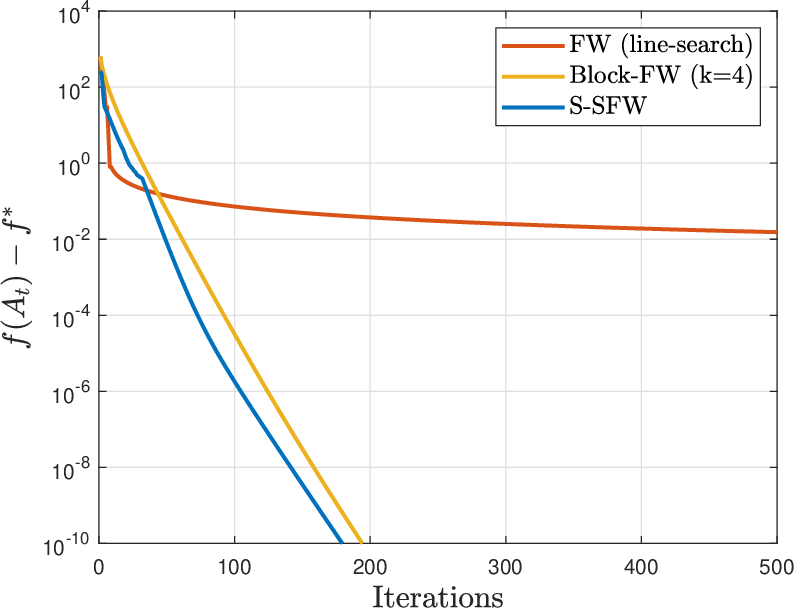}\\
    {\footnotesize (a) $\tau=0.01$}
  \end{minipage}\hfill
  \begin{minipage}[t]{0.48\textwidth}
    \centering
    \includegraphics[width=\linewidth]{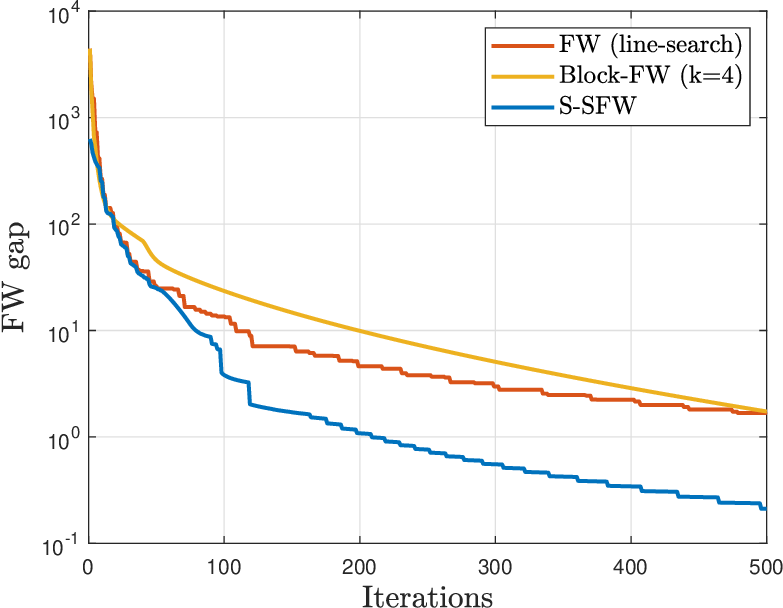}\\
    \includegraphics[width=\linewidth]{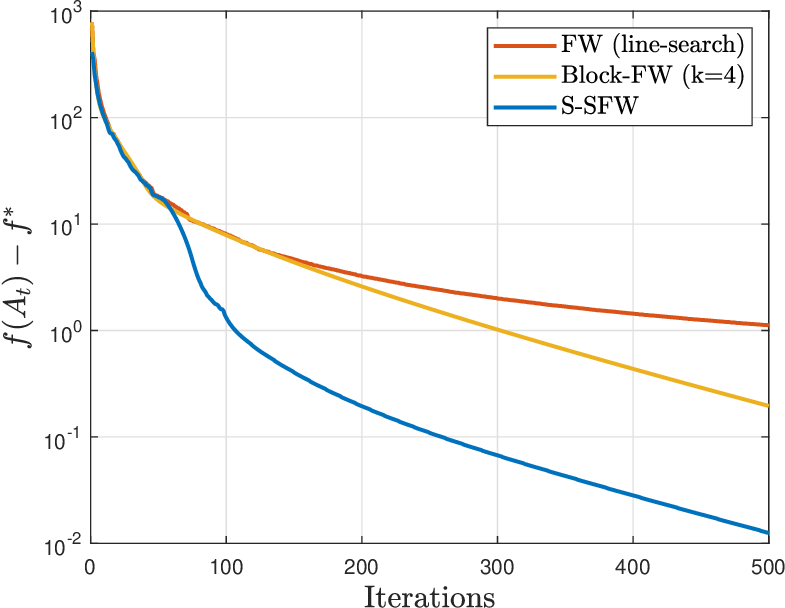}\\
    {\footnotesize (b) $\tau=0.03$}
  \end{minipage}
  \caption{\textsc{Mnist}: FW gap (top) and objective gap (bottom) for $\tau=0.01$ (a) and $\tau=0.03$ (b).}
  \label{fig: neural_compare}
\end{figure}

We evaluate two choices of $\tau$ and compare S-SFW against FW and
Block-FW.
Our implementation follows \cite{allen2017linear} at the benchmark level,
while using different training subset sizes and parameter choices.
For all methods, we set $L=2\tau^2Nd, \mu=0.4L$. 
For Block-FW, we fix the block size to $k=4$ in both settings, which matches the larger of the two optimal ranks (specifically, $r^*=2$ for $\tau=0.01$ and $r^*=4$ for $\tau=0.03$).
The experimental results are reported in Fig.~\ref{fig: neural_compare}. 
As shown in the figure, S-SFW consistently outperforms both FW and Block-FW across the two choices of $\tau$, demonstrating faster convergence in practice.

\section{Conclusions}\label{sec:conclusions}
This paper advances Frank--Wolfe methods for convex optimization over the spectrahedron by introducing a nonconvex Spectral Oracle based on spectral balls, which lift the simplex-ball geometry of~\cite{wang2025simplex} to the matrix setting.
Despite the nonconvexity of spectral balls, we establish structural properties that yield a closed-form oracle and an efficient implementation of cost comparable to standard Frank--Wolfe iterations when the optimal rank satisfies $r^*\ll n$.
Building on this oracle, we develop a rank-adaptive Frank--Wolfe scheme that requires no prior knowledge of $r^*$ and, after incorporating a fully corrective low-dimensional update, attains linear convergence after a finite burn-in phase under quadratic growth and strict complementarity.
In particular, the eigen-component count never exceeds $r^*$ and eventually stabilizes at $r^*$, thereby closing the gap between low-rank efficiency and fast convergence for spectrahedron-constrained Frank--Wolfe methods.
To our best knowledge, it is a first block-FW method that has block size between $1$ and $r^*$ to enjoy a
linear convergence rate.
Numerical experiments on quadratic sensing, matrix completion, and polynomial neural network training corroborate the theory and demonstrate competitive, robust performance relative to classical Frank--Wolfe and block variants.

\section*{Acknowledgments}
This work was supported by the National Natural Science Foundation of China (No.~12571323) and Hong Kong RGC General Research Fund (No.~PolyU/15303124).

\bibliographystyle{plainnat}
\bibliography{references}
\end{document}